\documentclass[a4paper,11pt,english,reqno]{amsart}
\usepackage{babel}
\usepackage{tikz}
\usepackage{hyperref}
\usepackage{here}
\usepackage{amssymb}
\usepackage{stmaryrd}
\usepackage{pdflscape}
\usepackage[graphicx]{realboxes}
\usepackage[figuresright]{rotating}
\usetikzlibrary{arrows,decorations.markings,positioning}
\tikzset{inner sep=0pt, node distance=5mm,
  root/.style={circle,draw,minimum size=5pt,thick},
  broot/.style={circle,draw,minimum size=5pt,thick,fill},
  xroot/.style={circle,draw,minimum size=5pt,thick,fill=gray!70!white},
  crossroot/.style={circle,draw,minimum size=5pt,thick,label=below:$\times$},
  doublearrow/.style={postaction={decorate},   decoration={markings,mark=at position .6 with {\arrow[line width=1.2pt]{>}}},double distance=1.6pt,thick},
  doublenoarrow/.style={double distance=1.6pt,thick},
  rdoublearrow/.style={postaction={decorate},   decoration={markings,mark=at position .4 with {\arrowreversed[line width=1.2pt]{>}}},double distance=1.6pt,thick},
  rtriplearrow/.style={postaction={decorate},   decoration={markings,mark=at position .4 with {\arrowreversed[line width=1.2pt]{>}}},double distance=2.5pt,thick},
  curvedline/.style={bend=right}
}

\newtheorem{theorem}{Theorem}[section]
\newtheorem{lemma}[theorem]{Lemma}
\newtheorem{proposition}[theorem]{Proposition}

\theoremstyle{definition}
\newtheorem{definition}[theorem]{Definition}
\newtheorem{example}[theorem]{Example}

\theoremstyle{remark}
\newtheorem{remark}[theorem]{Remark}

\numberwithin{equation}{section}

\newcommand{\C}{\mathbb{C}}

\newcommand{\Z}{\mathbb{Z}}
\newcommand{\Cl}{C\!\ell}
\newcommand{\Id}{\mathrm{Id}}
\newcommand{\so}{\mathfrak{so}}
\newcommand{\g}{\mathfrak{g}}

\newcommand{\h}{\mathfrak{h}}
\renewcommand{\L}{\mathcal{L}}
\newcommand{\gL}{\mathfrak{L}}

\renewcommand{\t}{\mathfrak{t}}
\renewcommand{\l}{\mathfrak{l}}
\renewcommand{\k}{\mathfrak{k}}
\renewcommand{\r}{\mathfrak{r}}
\newcommand{\mcB}{\mathcal{B}}
\newcommand{\p}{\mathfrak{p}}
\newcommand{\s}{\mathfrak{s}}
\newcommand{\m}{\mathfrak{m}}

\newcommand{\ad}{\mathrm{ad}}

\newcommand{\vol}{\mathrm{vol}}

\DeclareMathOperator{\der}{der}
\DeclareMathOperator{\gr}{gr}

\begin{document}
\title[{Prolongations of super-Poincar\'e algebras}]
{Classification of maximal transitive prolongations of super-Poincar\'e algebras}
\author[A.~Altomani]{Andrea Altomani}
\address[A.~Altomani]{Research Unit in Mathematics\\
Universit\'e du Luxembourg\\
6 rue Richard Coudenhove-Kalergi\\
L-1359 Luxembourg.
}
\email{andrea.altomani@uni.lu}
\author[A.~Santi]{Andrea Santi}
\address[A.~Santi]{Research Unit in Mathematics\\
Universit\'e du Luxembourg\\
6 rue Richard Coudenhove-Kalergi\\
L-1359 Luxembourg.
}\email{andrea.santi@uni.lu, asanti.math@gmail.com}
\thanks{\\ \phantom{ccc}The second author was supported by the project F1R-MTH-PUL-08HALO-HALOS08 of the University of Luxembourg.}


\begin{abstract} 
Let $V$ be a complex vector space with a non-degenerate symmetric bilinear form and $\mathbb S$ an
irreducible module over the Clifford algebra $\Cl(V)$ determined by this form. A supertranslation algebra is a $\Z$-graded
Lie superalgebra $\m=\m_{-2}\oplus\m_{-1}$, where $\m_{-2}=V$ and $\m_{-1}=\mathbb S\oplus\cdots\oplus\mathbb{S}$ is the direct sum of an arbitrary number $N\geq 1$ of copies of $\mathbb S$, whose bracket $[\cdot,\cdot]|_{\m_{-1}\otimes \m_{-1}}:\m_{-1}\otimes\m_{-1}\rightarrow\m_{-2}$ is symmetric, $\so(V)$-equivariant and non-degenerate (that is the condition ``$s\in\m_{-1}, [s,\m_{-1}]=0$'' implies $s=0$). 
We consider the maximal transitive prolongations in the sense of Tanaka of supertranslation algebras. We prove that they are finite-dimensional for $\dim V\geq3$ and classify them
in terms of super-Poincar\'e algebras and appropriate $\Z$-gradings of simple Lie superalgebras.
\end{abstract}
\maketitle
\section{Introduction}
\label{Introduction}
The theory of Lie superalgebras became a mainstream topic of research during the '70s, the interest being mainly motivated by the problem of constructing supersymmetric field theories, in particular, supergravity \cite{FP, Nahm}. 
One of the first Lie superalgebras to be considered was the $D=4$ super-Poincar\'e algebra, obtained from the Poincar\'e algebra $\so(4)\!\niplus\C^{4}$ in dimension four by adding ``odd spinorial generators''. 

Similar Lie superalgebras were then defined for all possible dimensions as follows.
\vskip0.2cm\par\noindent
\subsubsection*{Super-Poincar\'e algebras and super-Poincar\'e structures on supermanifolds}\hfill
\vskip0.2cm\par
Let $V$ be a complex vector space of dimension $\dim V=D$ endowed with a non-degenerate symmetric bilinear form $(\cdot,\cdot)$ and $$\Cl(V)=\Cl(V)_{\overline{0}}\oplus\Cl(V)_{\overline{1}}$$ the Clifford algebra determined by this form, with its natural $\mathbb{Z}_2$-grading.
We adopt the conventions used in \cite{LM}, in particular, the product in $\Cl(V)$ satisfies $vu+uv=-2(v,u)\mathbf 1$ for any $v,u\in V$ and the natural inclusion of the orthogonal Lie algebra in the Clifford algebra is given by the map
$$
\gamma:\so(V)\rightarrow\Cl(V)\;,\;\;\ v\wedge u\mapsto \frac{1}{4}(vu-uv)\;,\;\;
$$
where $v\wedge u$ is the anti-symmetric endomorphism $(v,\cdot)u-(u,\cdot)v$. Note that some authors prefer to use different conventions, cf. Deligne's lectures \cite{DeligneS}.

Let $W$ be a module over the Clifford algebra $\Cl(V)$. The action of an element $c\in\Cl(V)$ on $s\in W$ will be denoted by $c\cdot s$. We will denote by $N\geq 1$ the number of irreducible $\Cl(V)$-components of $W$, counted with their multiplicities (we note that this convention is not universally adopted, and some authors use ``$N$'' to denote the number of irreducible $\so(V)$-components). 

If $\dim V$ is even, there exists only one irreducible $\Cl(V)$-module, up to equivalence; if $\dim V$ is odd, there exist two inequivalent irreducible $\Cl(V)$-modules, they are equivalent and irreducible under the action of $\so(V)$ \cite{LM}. 

For our purposes, it will be actually sufficient to fix one of these irreducible $\Cl(V)$-modules and denote it  with the symbol $\mathbb S$.  We call $\mathbb S$ the {\it spinor} representation of $\so(V)$ and adopt the convention that it is a {\it purely odd} supervector space, that is $\mathbb S=\mathbb S_{\bar 0}\oplus\mathbb S_{\bar 1}$ with $\mathbb S_{\bar 0}=(0)$ and $\mathbb S_{\bar 1}=\mathbb S$. 

In particular, any $\Cl(V)$-module $W$, when seen as an $\so(V)$-module, is the direct sum
$$
W=\underbrace{\mathbb S\oplus\cdots\oplus\mathbb{S}}_{N-\text{summands}}
$$
of spinor representations and therefore a purely odd supervector space.
\vskip0.2cm\par

In the following definition, and throughout the paper, anti-symmetric and symmetric tensors
associated to a supervector space $U=U_{\bar 0}\oplus U_{\bar 1}$ are understood in the non-super sense, that is they are defined as
$$
\Lambda (U)=\bigotimes (U)/\langle x\otimes y+y\otimes x\,|\,x,y\in U\rangle
$$
and
$$
\,\,\,S(U)=\bigotimes (U)/\langle x\otimes y-y\otimes x\,|\,x,y\in U\rangle\,,
$$
even when $U$ has a non-trivial odd-part $U_{\bar 1}$.

\begin{definition}
\label{lelele}
A complex Lie superalgebra $\p=\p_{\bar0}\oplus\p_{\bar1}$ is called a {\it super-Poincar\'e algebra} if
\begin{itemize}
\item[--] $\p_{\bar0}=\so(V)\niplus V$ is the usual Poincar\'e algebra associated to $V$;
\item[--] $\p_{\bar1}=W$ is the direct sum of an arbitrary number $N\geq 1$ of copies of the spinor representation;
\item[--] the natural action of $\so(V)$ on $W$ is extended to all $\p_{\bar0}$ by $[V,W]=0$;
\item[--] the bracket between odd elements takes values in $V$ and it is given by a symmetric $\so(V)$-equivariant bilinear map
\begin{equation}
\label{eqgamma}
\Gamma:S^2(W)\rightarrow V
\end{equation}
satisfying the following non-degeneracy condition: if $s\in W$ is an element such that $\Gamma(s,W)=0$, then $s=0$. 
\end{itemize}
\end{definition}
A complete classification of super-Poincar\'e algebras was achieved in the '90s by Alekseevsky and Cort\'es: the main result of \cite{AC} is indeed an explicit description of a basis of the space of $\so(V)$-invariant elements in $S^2(W^*)\otimes V$. 

Let us recall this description. Following \cite{AC}, a non-degenerate bilinear form $\mathcal{B}\colon W\otimes W\to\C$ is called \emph{admissible} if there exist $\tau,\sigma\in\{\pm1\}$ such that 
$$\mcB(v\cdot s, t)=\tau \mcB(s, v\cdot t)=\sigma \mcB(t, v\cdot s)$$ for all $v\in V$ and $s,t\in W$. In \cite{AC}, it is proved that the space of admissible bilinear forms on $W$ is always non-trivial.

Consider now an admissible bilinear form such that $\sigma\tau=1$. It satisfies the following properties:
\begin{enumerate}
\item[(B1)]
$\mcB$ is $\so(V)$-invariant,
\item[(B2)]
$\mcB$ is symmetric or anti-symmetric 
(we let $\epsilon=1$ in the former case and $\epsilon=-1$ in the latter),
\item[(B3)]
for all $v\in V$ and $s,t\in W$, $\mcB(v\cdot s,t)=\epsilon\mcB(s,v\cdot t)$.
\end{enumerate}
One can easily deduce from (B1)-(B3) that the bilinear map $\Gamma:W\otimes W\rightarrow V$ defined by
\begin{equation}
\label{equazione_bracket}
(\Gamma(s,t),v)=\mcB(v\cdot s,t)\qquad\text{for all}\,\,v\in V\,\,\text{and}\,\,s,t\in W
\end{equation}
corresponds to a super-Poincar\'e algebra, that is it is symmetric, $\so(V)$-equivariant and non-degenerate in the sense of Def. \ref{lelele}. 

The main result of \cite{AC} is that the space of $\so(V)$-invariant elements in $S^2(W^*)\otimes V$ has a basis consisting of tensors \eqref{eqgamma} which are associated, in the above manner, to admissible bilinear forms satisfying $\sigma\tau=1$.
\vskip0.3cm\par\noindent
{\it From now on any super-Poincar\'e algebra is tacitly assumed to be determined by an admissible bilinear form such that $\sigma\tau=1$.}
\vskip0.3cm\par
Super-Poincar\'e algebras admit a natural realization as algebras of super-vector fields on supermanifolds.

Recall that a complex supermanifold of dimension $(m|n)$ is a pair 
\begin{equation}
\label{cenci}
\mathbb M=(M,\mathcal A(M))\,,
\end{equation}
formed by an $m$-dimensional complex manifold $M$ (called the \emph{body}) and a sheaf of superalgebras $\pi:\mathcal A(M)\rightarrow M$ (called the \emph{sheaf of superfunctions}) such that for any point $x\in M$ there exist an open neighbourhood $U\supset\{x\}$ and an isomorphism of sheaves of superalgebras
$$
\mathcal A(M)|_{U}\simeq \mathcal O(M)|_{U}\otimes \Lambda((\mathbb{C}^{n})^{*})\,,
$$
where $\mathcal O(M)$ denotes the sheaf of holomorphic functions of $M$, see \cite{Ma, QFS}. 
Local holomorphic coordinates $(z^i)$ on $U$ and generators $(\xi^\alpha)$ of the Grassmann algebra $\Lambda((\mathbb C^{n})^{*})$ are respectively called \emph{even coordinates} and \emph{odd coordinates} of the supermanifold \eqref{cenci}.

It is then easy to see that any super-Poincar\'e algebra 
$$
\p=\p_{\bar0}\oplus\p_{\bar1}\,,\;\;\text{where}\;\;\p_{\bar0}=\so(V)\niplus V\;\;\text{and}\;\;\p_{\bar 1}=W\,,
$$
admits a natural realization as an algebra of super-vector fields
\begin{align*}
&\frac{\partial}{\partial z^i}&&\text{that span}\; V,\\
&\frac{\partial}{\partial\xi^\alpha}-\frac{1}{2}\xi^{\beta}\Gamma_{\beta\alpha}^i\frac{\partial}{\partial z^i}&&\text{that span}\;W,\\
&z^i\frac{\partial}{\partial z^j}-z^j\frac{\partial}{\partial z^i}+\gamma_{ij\alpha}^{\;\;\,\beta}\,\xi^\alpha\frac{\partial}{\partial\xi^\beta}&&\text{that span}\; \so(V),
\end{align*}
on the $(D|N\!\dim\mathbb{S})$-dimensional linear supermanifold
\begin{equation}
\label{cenci2}
\mathbb M=(\mathbb C^D, \mathcal{O}(\mathbb C^D)\otimes \Lambda(W^*))\ .
\end{equation} 

By assigning degrees 
$$\deg z^i=2=-\deg\frac{\partial}{\partial z^i}\,\,\,\,\,\,\text{and}\,\,\,\,\,\,\deg\xi^\alpha=1=-\deg\frac{\partial}{\partial\xi^\alpha}\,,$$
the super-Poincar\'e algebra inherits a natural $\Z$-grading
$\p=\bigoplus_{p\in\mathbb Z}\p_p$, where
$$\p_{-2}=V,\,\,\,\,\p_{-1}=W,\,\,\,\,\p_{0}=\so(V)\,\,\,\,\text{and}\,\,\,\,\p_{p}=0\,\,\,\,\text{for all}\,\,\, p\neq -2, -1, 0,$$ 
which satisfies the usual property $
[\p_p,\p_q]\subset\p_{p+q}$ for all $p,q\in\mathbb Z$.
In this case an additional compatibility condition between the $\mathbb Z$-grading and the Lie superalgebra structure holds true. Indeed the even part of $\p$ is the direct sum of the homogeneous subspaces of even degree, and similarly the odd part:
\[
\p_{\bar 0}=\bigoplus_{p\in\mathbb Z}\p_{2p},\qquad
\p_{\bar 1}=\bigoplus_{p\in\mathbb Z}\p_{2p+1}.
\]
In other words, the parity of $\p=\p_{\bar 0}\oplus\p_{\bar 1}$ concides with the $\Z$-grading $(\mathrm{mod}\,2)$; gradings with this property are called \emph{consistent}.
\vskip0.2cm\par
The negatively graded parts of super-Poincar\'e algebras are usually called {\it supertranslation algebras}. 
Explicitly, one has the following.
\begin{definition}

A consistently $\mathbb{Z}$-graded Lie superalgebra $\m=\m_{\bar 0}\oplus\m_{\bar 1}$ with 
$$
\m_{\bar0}=\m_{-2}=V\,\,\,\,\,\,\text{and}\,\,\,\,\,\,\m_{\bar1}=\m_{-1}=W
$$
is called a {\it supertranslation algebra} if 
the bracket between odd elements is given by a tensor $\Gamma:S^2(W)\rightarrow V$ which is of the form \eqref{equazione_bracket} for some admissible bilinear form $\mcB$ on $W$ such that $\sigma\tau=1$.
\end{definition}
The simply connected nilpotent Lie supergroup corresponding to a supertranslation algebra $\m=\m_{-2}\oplus\m_{-1}$
is clearly identifiable with the linear supermanifold \eqref{cenci2}
and one can associate to $\m_{-1}\subset T_{e}\mathbb M$ a unique $\p$-invariant distribution $\mathcal{D}$ on $\mathbb M$ \cite{SaS2, SaHom}. 

The distribution has depth $2$ 
and its Levi form $$\mathcal{L}_x:S^2(\mathcal{D}_x)\rightarrow T_x\mathbb{M}/\mathcal{D}_x\,,\qquad (X,Y)\mapsto [X,Y]_x\,\pmod{\mathcal{D}_x}$$
is identifiable with the tensor $\Gamma$ at any point $x$ of the body of $\mathbb M$. In particular, non-degeneracy of $\Gamma$ implies that $\mathcal{D}$ is maximally non-integrable.
\vskip0.2cm\par
In \cite{SaS1, SaS2, AlSa} the following notion of curved analogs of the homogeneous models $(\mathbb M, \mathcal{D})$ was considered.
\begin{definition}
\label{defin}
Let $\m=\m_{-2}\oplus\m_{-1}$ be a fixed supertranslation algebra. A {\it super-Poincar\'e structure} (of type $\m$) on a 
supermanifold $\mathbb M$ of dimension $$\dim \mathbb M = (\dim \m_{-2}|\dim \m_{-1})$$
is the datum of a depth $2$ distribution $\mathcal{D}$ with $\operatorname{rank}\mathcal{D} = \dim \m_{-1}$ whose Levi form
$\mathcal{L}_x$ is identifiable at all points $x$ of the body of $\mathbb M$ with the tensor $\Gamma$ corresponding to $\m$.
\end{definition}

Our motivation to study super-Poincar\'e structures relies on the interesting fact that supergravity theories admit, besides the traditional ``component formalism'' formulations (see \cite{FP, Nahm}), more geometric presentations in terms of super-Poincar\'e structures $(\mathbb M,\mathcal{D})$ (see \cite{BGLS, GL, L, RKS, RS1, RS2, SaS1, SaS2}). 

The physical fields of the component formalism presentation, as well as the equations they satisfy, can be represented by appropriate tensorial objects on the supermanifold $\mathbb M$ and the supersymmetry transformations by Lie derivatives along sections of the distribution $\mathcal{D}$.
\vskip0.2cm\par\noindent
\subsubsection*{Maximal transitive prolongations of supertranslation algebras}\hfill
\vskip0.2cm\par
The {\bf main result} Theorem \ref{N1c} of this work  is the explicit description of the maximal transitive prolongation $\g$ in the sense of N. Tanaka \cite{N2} of a supertranslation algebra $\m$, for all possible dimensions $D$ and all $N\in\mathbb{N}$. 
In geometrical terms, $\g$ is the algebra of all infinitesimal symmetries of the homogeneous model $(\mathbb M,\mathcal{D})$ corresponding to $\m$.
\vskip0.2cm\par
In order to formulate our main result, we recall that the \emph{maximal transitive prolongation} of a negatively graded fundamental Lie superalgebra
$$
\m=\bigoplus_{-d\leq p\leq -1}\m_p\qquad(\m\,\text{is negatively graded of depth $d$})
$$
$$
\m_{p}=[\m_{-1},\m_{p+1}]\,\,\,\text{for all}\,\,\,p\leq -2\qquad(\m\,\,\text{is fundamental})\,\,\,\,
$$
is a (possibly infinite-dimensional) $\mathbb{Z}$-graded Lie superalgebra 
$$\g=\bigoplus_{p\in \Z} \g_p$$ 
such that:
\begin{enumerate}
\item
$\g_p$ is finite-dimensional for every $p\in\Z$;
\item
$\g_p=\m_p$ for every $-d\leq p\leq -1$ and $\g_p=0$ for every $p<-d$;
\item
for all $p\geq 0$, if $D\in\g_p$ is an element such that $[D,\g_{-1}]=0$, then $D=0$ (\emph{transitivity});
\item
$\g$ is maximal with these properties, i.e., if $\g'$ is another $\mathbb{Z}$-graded Lie superalgebra satisfying (1), (2), and (3), then there exists an injective homomorphism of $\Z$-graded Lie superalgebras $\phi\colon\g'\to\g$.
\end{enumerate}
The existence and uniqueness of $\g$ is proved in \cite{N2} (the proof is given in the Lie algebra case but it extends verbatim to the superalgebra case).
A concise and self-contained presentation of $\g$ using partial differential equations can also be found in \cite{Schepo}.

Note that, by transitivity, the maximal transitive prolongation of a consistently $\Z$-graded Lie superalgebra $\m$ is also consistently $\Z$-graded, that is 
$$\g_{\overline 0}=\bigoplus_{p\in\Z}\g_{2p}\;\;\;\;\text{and}\;\;\;\;\g_{\overline 1}=\bigoplus_{p\in\Z}\g_{2p+1}\ .$$
\vskip0.2cm\par
Consider now a supertranslation algebra $\m=\m_{-2}\oplus\m_{-1}$, where $\m_{-2}=V$ and $\m_{-1}=W$, together with its maximal transitive prolongation $\g$. The main results of this paper are now illustrated.\vskip0.2cm\par

If $\dim V=1$, the maximal transitive prolongation $\g$ is infinite-dimensional and 
isomorphic to the contact Lie superalgebra $K(1|N)$ described in \cite{Kac98}; indeed a supermanifold of dimension $(1|N)$ endowed with a super-Poincar\'e structure is just a contact supermanifold. If $\dim V=2$, then $\g$ is the direct sum of two copies of $K(1|N)$.  

On the other hand, the following result holds.
\begin{theorem}
If $\dim V\geq 3$, the maximal transitive prolongation of a supertranslation algebra is finite-dimensional.
\end{theorem}
A similar theorem was proved in the Lie algebra case \cite{AlSa} by applying a deep result of N. Tanaka and J.-P. Serre \cite{N2, GS} (see Remark \ref{Serre}). The Lie superalgebra analogue of this deep result is not valid; our proof of finite-dimensionality uses different techniques and relies on the classification of the infinite-dimensional simple linearly compact Lie superalgebras \cite{Kac98}. 
\vskip0.2cm\par
In all cases except those listed in Theorem \ref{teorematotale}, Table \ref{tabmax}, the prolongation $\g$ satisfies $\g_{p}= 0$ for all $p\geq 1$ or equivalently it is the vector space direct sum
$$\g=\p\oplus\C E\oplus\h_0\,,$$ 
where $\p=\m\inplus\so(V)$ is the super-Poincar\'e algebra corresponding to $\m$,
$$E=\sum_{\alpha=1}^{N\dim \mathbb S}\xi^{\alpha}\frac{\partial}{\partial  \xi^{\alpha}}+2\sum_{i=1}^{D}z^{i}\frac{\partial}{\partial  z^{i}}$$
the Euler vector field and 
$\h_0=\{D\in \g_0| [D,{\m_{-2}}]=0\}$ 
the algebra of internal symmetries of $\m_{-1}=W$.

The cases where the positively graded part of $\g$ is not trivial are listed in Theorem \ref{teorematotale}, Table \ref{tabmax}, and reproduced in Table \ref{tabintroduction}  for reader's convenience. Therein, and throughout the paper, the symbol "$\cdots$" appearing in a Dynkin diagram of a Lie superalgebra will denote a subdiagram corresponding to a Lie algebra $\mathfrak{sl}(\ell+1)$ of appropriate (possibly zero) rank $\ell$. Finally,  finite-dimensional simple Lie superalgebras are denoted according to the conventions used in, e.g., \cite{Serga,Leites} and $\mathfrak{pg}=\g/\mathbb{C}\Id$ is the projectivization of any linear Lie superalgebra $\g$ which contains the scalar matrices.
  
{\small
\begin{table}[H] 
\begin{centering}
\begin{tabular}{|c|c|c|c|c|c|}
\hline
$\g$ & Dynkin diagram & $\dim V$ & $\dim W$ & $N$ & $\h_0$\\
\hline
\hline
&&&&&\\[-3mm]

$\mathfrak{osp}(1|4)$&
\begin{tikzpicture}
\node[root]   (3) [label=${{}^2}$]      {} ;
\node[broot]   (4) [right=of 3] {} [label=${{}^2}$]     edge [-] (3);
\end{tikzpicture}&
$3$ & $2$ & $1$ & $0$
\\
\hline
&&&&&\\[-3mm]

$\begin{gathered}\mathfrak{osp}(2m+1|4)\\ m\geq1\end{gathered}$&
\begin{tikzpicture}
\node[root]   (3) [label=${{}^2}$] {};
\node[xroot]   (4) [right=of 3] {} [label=${{}^2}$]  edge [-] (3);
\node[]   (5) [right=of 4] {$\;\cdots\,$} edge [-] (4);
\node[root]   (7) [right=of 5] {} [label=${{}^2}$]  edge [rdoublearrow] (5);
\end{tikzpicture}&
$3$ & $2N$ & $2m+1$ & $\so(2m+1)$
\\
\hline 
&&&&&\\[-3mm]

$\mathfrak{osp}(2|4)$&
\begin{tikzpicture}
\node[root]   (3) [label=${{}^2}$] {} ;
\node[xroot]   (4) [above right=of 3] [label=right:${{}^1}$] {} edge [-] (3);
\node[xroot]   (5) [below right=of 3] [label=right:${{}^1}$] {} edge [-] (3) edge [doublenoarrow] (4);
\end{tikzpicture}&
$3$ & $2N$ & $2$ & $\so(2)$
\\
\hline 
&&&&&\\[-3mm]

$\begin{gathered}\mathfrak{osp}(2m|4)\\ m\geq2\end{gathered}$&
\begin{tikzpicture}
\node[root]   (3)[label=${{}^2}$] {} ;
\node[xroot]   (4) [right=of 3][label=${{}^2}$] {} edge [-] (3);
\node[]   (6) [right=of 4] {$\;\cdots\quad$} edge [-] (4);
\node[root]   (7) [above right=of 6] [label=${{}^1}$]{} edge [-] (6);
\node[root]   (8) [below right=of 6] [label=${{}^1}$]{} edge [-] (6);
\end{tikzpicture}&
$3$ & $2N$ & ${2m}$ & $\so(2m)$
\\
\hline
&&&&&\\[-3mm]

$\begin{gathered}\mathfrak{sl}(m+1|4)\\m\neq3\end{gathered}$&
\begin{tikzpicture}
\node[root]   (1)       [label=${{}^1}$]              {};
\node[xroot] (2) [right=of 1] {}[label=${{}^1}$] edge [-] (1);
\node[]   (3) [right=of 2] {$\;\cdots\,$} edge [-] (2);
\node[xroot]   (4) [right=of 3] {}[label=${{}^1}$] edge [-] (3);
\node[root]   (5) [right=of 4] {}[label=${{}^1}$] edge [-] (4);
\end{tikzpicture}&
$4$ & $4N$ & $m+1$ & $\mathfrak{gl}(m+1)$
\\
\hline
&&&&&\\[-3mm]

$\mathfrak{pgl}(4|4)$&
\begin{tikzpicture}
\node[root]   (1)       [label=${{}^1}$]              {};
\node[xroot] (2) [right=of 1] {}[label=${{}^1}$] edge [-] (1);
\node[]   (3) [right=of 2] {$\;\cdots\,$} edge [-] (2);
\node[xroot]   (4) [right=of 3] {}[label=${{}^1}$] edge [-] (3);
\node[root]   (5) [right=of 4] {}[label=${{}^1}$] edge [-] (4);
\end{tikzpicture}&
$4$ & $4N$ & $4$ & $\mathfrak{gl}(4)$
\\
\hline
&&&&&\\[-3mm]

$\mathfrak{ab}(3)$&
\begin{tikzpicture}
\node[root]   (1)        [label=${{}^1}$]             {};
\node[xroot] (2) [right=of 1] {} [label=${{}^2}$]  edge [rtriplearrow] (1) edge [-] (1);
\node[root]   (3) [right=of 2] {} [label=${{}^3}$] edge [-] (2);
\node[root]   (4) [right=of 3] {} [label=${{}^2}$] edge [doublearrow] (3);
\end{tikzpicture}&
$5$ & $4N$ & $2$ & $\mathfrak{sl}(2)$\\
\hline 
\end{tabular}
\end{centering}
\caption[]{\label{tabintroduction}}
\vskip10pt
\end{table}  
}

The negatively graded parts $\g_{<0}=\bigoplus_{p<0}\g_{p}$ of the Lie superalgebras listed in Table \ref{tabintroduction} are supertranslation algebras and their explicit description 
is provided in Examples \ref{special}, \ref{orto}, and \ref{effe}. 
The Lie superalgebras $\mathfrak{osp}(N|4)$ and $\mathfrak{sl}(N|4)$ with the $\Z$-grading described in Table \ref{tabintroduction} already appeared in \cite{L}.

Also the $\Z$-graded Lie superalgebra
{\small
\begin{table}[H]
\begin{centering}
\begin{tabular}{|c|c|c|c|c|}
\hline
$\g$ & Dynkin diagram & $\g_{-2}$ & $\g_{-1}$ & $\g_0$\\
\hline
\hline
&&&&\\[-3mm]
$\begin{gathered}\mathfrak{osp}(8|2n)\\ n\geq1\end{gathered}$&
\begin{tikzpicture}
\node[root]   (1)    [label=${{}^1}$]                 {};
\node[root] (2) [right=of 1] [label=${{}^2}$]{} edge [-] (1);
\node[root] (3) [right=of 2] [label=${{}^2}$]{} edge [-] (2);
\node[xroot] (4) [right=of 3] [label=${{}^2}$]{} edge [-] (3);
\node[]   (5) [right=of 4] {$\;\cdots\,$} edge [-] (4);
\node[root]   (6) [right=of 5] [label=${{}^1}$]{} edge [doublearrow] (5);
\end{tikzpicture}&
$\C^6$ & $\mathbb{S}^+\otimes \C^{2n}$ & $\so(6)\oplus\C E\oplus \mathfrak{sp}(2n)$
\\
\hline 
\end{tabular}
\end{centering}
\caption[]{\label{tabintro2}}   
\vskip10pt
\end{table}}
\par\noindent
has a non-trivial positively graded part; its negatively graded part is described in Example \ref{notfit} and, although not isomorphic to a supertranslation algebra, admits a similar description in terms of semi-spinor representations $\mathbb S^+$ in dimension $\dim V=6$ (if $\dim V$ is even,  the spinor representation $\mathbb S$ is not $\so(V)$-irreducible and it decomposes $\mathbb S=\mathbb S^+\oplus\mathbb S^-$ into two irreducible components, called {\it semi-spinor} representations).
\vskip0.2cm\par

The existence of a non-trivial positively graded part of $\g$ has a geometrical significance: it provides additional local symmetries of the homogeneous model $\mathbb M$.
In this case, the inclusion of $\m$ in $\g$ induces an open dense embedding of the model into the flag supermanifold $\overline{ \mathbb M}=G/G_{\geq 0}$, where $G$ denotes the simply connected Lie supergroup with Lie superalgebra $\g$ as in Tables \ref{tabintroduction} and \ref{tabintro2} and $G_{\geq 0}$ the parabolic subsupergroup associated to $\g_{\geq 0}=\bigoplus_{p\geq 0}\g_{p}$. 

We believe that the existence of a non-trivial positively graded part is responsible for the off-shell nature of the supergravity theories in $\dim V\leq 6$ modeled on the Lie superalgebras of Examples \ref{special}, \ref{orto},  \ref{effe}, \ref{notfit} (in physics, a theory has an {\it off-shell formulation} if its symmetries are well-defined on the whole set of fields which do not
necessarily satisfy the equations of motion; only such theories are suitable for quantisation). Indeed a crucial step in a super-space formulation of supergravity theories is the choice of a class of connections compatible with the distribution $\mathcal{D}$ and satisfying appropriate torsion contraints, see \cite{L, RKS, SaS1, SaS2}; for $\dim V\leq 6$ the torsion constraints do not determine the connection uniquely. Deeper investigations of this subject will be the content of a future work.
\vskip0.2cm\par\noindent
\subsubsection*{Structure of the paper}\hfill
\vskip0.2cm\par

The paper is structured as follows. In Section \ref{sec1}  we 
adapt to the Lie superalgebra case some of the results already proved in \cite{AlSa} for Lie algebras, giving, in particular, Theorem \ref{lo0} on the structure of $\g_0$.

Section \ref{secfin} specializes to the case $\dim V\geq 3$ and contains most of the technical results of the paper. Therein we first prove that the maximal transitive prolongation $\g$ either satisfies $\g_p=0$ for all $p\geq 1$, or is semisimple and contains a unique minimal ideal $\s$, which is a simple prolongation of $\m$ (Theorem \ref{simpleprol}). Then, with the help of the classification of the $\Z$-graded even transitive irreducible infinite-dimensional Lie superalgebras and strongly transitive modules \cite{Kac98}, we show that $\s$ and $\g$ are finite-dimensional.

In Section \ref{sec2} we first classify all simple Lie superalgebras $\s$ that arise as prolongations of a supertranslation algebra (Theorem \ref{classimple}). They  are all \emph{basic} Lie superalgebras (i.e., not of Cartan type or belonging to the \emph{strange} series $\mathfrak{spe}$ and $\mathfrak{psq}$) and their $\Z$-gradings are given in terms of Dynkin diagrams. Theorem \ref{teorematotale} then describes the maximal transitive prolongations $\g$ whose minimal ideal $\s$ is one of the Lie superalgebras of Theorem \ref{classimple}.

Section \ref{sec3} contains the full classification result (Theorem \ref{N1c}) and the final Section \ref{concl} is devoted to the comparison of the results of this paper with the corresponding ones obtained in the Lie algebra case in \cite{AlSa}.

\vskip0.2cm\par\noindent
\subsubsection*{Notations}\hfill
\vskip0.2cm\par
Given any supervector space $U=U_{\bar 0}\oplus U_{\bar 1}$, we denote by 
$$\Pi U=(\Pi U)_{\bar 0}\oplus(\Pi U)_{\bar 1}$$ the supervector space with opposite parity, that is 
\begin{equation}
\label{pcf}
(\Pi U)_{\bar 0}=U_{\bar 1}\;\;,\qquad(\Pi U)_{\bar 1}=U_{\bar 0}
\end{equation}
as (non-super) vector spaces. 

The tensor product $U\otimes U'$ of two supervector spaces has a natural structure of supervector space given by
$$
(U\otimes  U')_{\overline 0}=(U_{\overline 0 }\otimes U'_{\overline 0})\oplus (U_{\overline 1}\otimes U'_{\overline 1})\,,\,(U\otimes  U')_{\overline 1}=(U_{\overline 0 }\otimes U'_{\overline 1})\oplus (U_{\overline 1}\otimes U'_{\overline 0})\ .
$$
For any positive integer $m$, we denote by $U^{\otimes m}$ the tensor product of $m$-copies of a supervector space $U$.

Finally, the reader should not confuse semi-spinor with {\it half-spin} representations. The latter are the finite-dimensional representations of $\so(V)$ which integrate to $\mathrm{Spin}(V)$ but do not integrate to $\mathrm{SO}(V)$. Half-spin representations are characterized as direct sums of $\so(V)$-submodules of $\bigoplus_{n\in\mathbb N}\mathbb S^{\otimes 2n+1}$.
\vskip0.2cm\par\noindent
{\it Acknowledgements}\/. {The authors are grateful to the anonymous referee for his helpful comments and suggestions.}
\vskip0.3cm\par\noindent

\section{Preliminary results}
\label{sec1}
In this section we adapt to the Lie superalgebra case some of the results already proved in \cite{AlSa} for Lie algebras: we explicitly describe the subalgebra $\g_0$ of the maximal transitive prolongation $\g=\bigoplus_{p\in\Z}\g_p$ of a supertranslation algebra $\m=V\oplus W$, construct an equivariant embedding of $\g_1$ into $W$, and prove that the action of the positively graded part of $\g$ on $\g_{-2}$ is faithful. The results of this section will be frequently and tacitly used throughout the paper.
\vskip0.3cm
\par
We recall that $\g_0$ is the Lie algebra of $0$-degree derivations of $\m$. We will use as synonyms the notations $[D,X]$ and $DX$ to denote the bracket in $\g$ of an element $D\in\g_0$ and an element $X\in\m$; we will also tacitly identify the spaces $\m_{-1}$, $\g_{-1}$ and $W$ (resp. $\m_{-2}$, $\g_{-2}$ and $V$).

Let now $E\in\g_0$ be the derivation acting with eigenvalues $-1$ on $\m_{-1}$ and $-2$ on $\m_{-2}$. We call $E$ the \emph{grading element} of $\g$.
Moreover, let  
\[\h_0=\{D\in\g_0\mid [D,\g_{-2}]=0\}\]
be the set of elements in $\g_0$ acting trivially on $\g_{-2}$.
\vskip0.3cm\par
We quote now some results from \cite{AlSa}, whose proofs carry over unchanged to the Lie superalgebra case.
\begin{theorem}[\protect{\cite[Theorem 2.3]{AlSa}}]
\label{lo0}
The Lie algebra $\g_0$ is a direct sum of ideals:
\[
\g_0=\so(V)\oplus\C E\oplus\h_0\,,
\]
where $\so(V)$ acts on $\m_{-2}=V$ (resp. $\m_{-1}=W$) via the tautological representation (resp. a multiple of the spinor representation).\qed 
\end{theorem}

\begin{lemma}[\protect{\cite[Lemma 2.5]{AlSa}}]
\label{lemmag1}
There exists a unique $\so(V)$-equivariant linear map $\phi\colon\g_1\to W$ satisfying
\[
Dv=v\cdot \phi(D)
\]
for all $D\in\g_1$ and $v\in V$.\qed
\end{lemma}

\begin{proposition}[\protect{\cite[Proposition 2.6]{AlSa}}]
\label{hanome}
For every $v\in V$, there exists a unique $0$-degree Lie superalgebra homomorphism $\psi_v:\g\to\g$ which satisfies
\begin{enumerate}
\item
$\psi_v(s)=v\cdot s$ for all $s\in W$;
\vskip0.1cm
\item for all $u\in V$,
$\psi_v(u)=\begin{cases} 
\epsilon(v,v)\left(u-\frac{2(v,u)}{(v,v)}v\right)& \text{if $v$ is non-isotropic,}\\
-2\epsilon(v,u)v& \text{if $v$ is isotropic;}
\end{cases}$
\vskip0.1cm
\item
$
\psi_v(\phi(\psi_v(D)))=\epsilon\phi(D)
$
for all $D\in \g_1$.
\end{enumerate}
\vskip0.1cm
Moreover, $\psi_v$ is invertible if and only if  $v$ is non-isotropic. \qed
\end{proposition}

The next result corresponds to \cite[Theorem 2.4 (2)]{AlSa}. There are however some differences between the classical and super case and it is appropriate to give 
a full proof.

\begin{proposition}
\label{finedopo}
Let $\g=\bigoplus_{p\in\mathbb{Z}}\g_p$ be the maximal transitive prolongation of a supertranslation algebra $\m=V\oplus W$ with $\dim V\geq 3$. For all $p\geq 1$, if $D\in\g_p$ and $[D,\g_{-2}]=0$, then $D=0$.
\end{proposition}
\begin{proof}
For every $p\geq 0$, it is known that the space $\{D\in\mathfrak{g}_{p}\,|\,[D,\mathfrak{g}_{-2}]=0\}$ is identifiable with the $p$-th term
\vskip-0.1cm
$$
\mathfrak{h}_0^{(p)}=(W\otimes\Lambda^{p+1}(W^*))\cap (\mathfrak{h}_0\otimes \Lambda^p(W^*))
$$
of the Cartan superprolongation (see, \textit{e.g.}, \cite{Gal} for more details) of the purely even Lie superalgebra $\mathfrak{h}_0\subset\mathfrak{gl}(W)$ acting on the purely odd supervector space $W$. 

Let $x,y,z\in V$ be orthogonal non-isotropic vectors and consider the bilinear form $\alpha$ on $W$ defined by 
\begin{equation}
\label{laforma}
\alpha(s,t)=([y\cdot z\cdot s,t],x)=\mcB(x\cdot y \cdot z \cdot s,t)
\end{equation}
for every $s,t\in W$. Straightforward computations show that 
\begin{enumerate}
\item
$\alpha$ is anti-symmetric and non-degenerate,
\item
$\h_0\subset \mathfrak{osp}(W,\alpha)$.
\end{enumerate}
Since $\mathfrak{osp}(W,\alpha)^{(p)}=0$ for every $p\geq 1$ (see, \textit{e.g.}, \cite[Theorem 5.1]{Gal}), also $\h_0$ has a Cartan superprolongation which is trivial in positive degrees. 
\end{proof}

Specialized to the case $p=1$, Proposition \ref{finedopo} asserts that the $\so(V)$-equivariant linear map 
$\phi\colon\g_1\to W$ considered in Lemma \ref{lemmag1} is injective, that is any $D\in\g_1$ is uniquely determined by its action on $V$. Moreover, by (3) of Proposition \ref{hanome}, the image $\phi(\g_1)$ is a $\Cl(V)$-submodule of $W$. 
\smallskip

\begin{remark}
\label{Serre}
In the Lie algebra case, \cite[Theorem 2.4]{AlSa} states in addition that the maximal transitive prolongation of an extended translation algebra $\m=V\oplus W$ with $\dim V\geq 3$ is {\it finite-dimensional}.

This is a consequence of a deep theorem of Tanaka \cite[Theorem 11.1]{N2} which is based on some arguments of Serre \cite{GS} on Spencer cohomology of Lie algebras. In its more general form, this deep theorem says: 
\vskip0.2cm\par\noindent
{\it The maximal transitive prolongation $\g$ of a fundamental Lie algebra $$\m=\bigoplus_{-d\leq p\leq -1}\m_{p}$$ is finite-dimensional if and only if the Cartan prolongation of the Lie algebra $$\h_0=\{D\in\g_0\mid [D,\bigoplus_{-d\leq p\leq -2}\m_{p}]=0\}\subset \mathfrak{gl}(\m_{-1})$$ is finite-dimensional.} 
\vskip0.2cm\par

The naive generalization of Tanaka's result is not true for Lie superalgebras. As a counterexample, consider the infinite-dimensional exceptional semisimple Lie superalgebra $\g=E'(5|10)$ described in \cite[\S 4.3]{CK99}. It is the maximal transitive prolongation of the consistently $\Z$-graded Lie superalgebra $\m=\m_{-2}\oplus\m_{-1}$, where $\m_{-2}=(\C^{5})^{*}$ and $\m_{-1}=\Pi(\Lambda^2(\C^5))$, with bracket given by $[\alpha,\beta]=\imath_{\alpha\wedge\beta}\vol$, for any $\alpha, \beta\in \m_{-1}$. The subalgebra $\g_{0}$ is $\mathfrak{gl}(5)$ acting in the obvious way on $\m$; in particular, $\h_{0}=0$.
\end{remark}

We will prove in Section \ref{secfin} that the maximal transitive prolongation $\g$ of a supertranslation algebra $\m=V\oplus W$ with $\dim V\geq 3$ is finite-dimensional. Our proof does not rely on a generalization of Tanaka's result but rather on the existence (when $\g_{1}\neq 0$) of a ``large'' simple ideal $\m\subset \s\subset \g$ (see Theorem \ref{simpleprol}) and on the classification of {\it $\Z$-graded even transitive irreducible infinite-dimensional Lie superalgebras} and of {\it strongly transitive modules} given in \cite{Kac98}.

In the low dimensional cases $\dim V=1,2$, 
the Lie superalgebra $\g$ is infinite-dimensional. These cases will be discussed in detail in Section \ref{sec3}.
\vskip0.3cm\par\noindent

\centerline
{\it From now on, we will assume that $\dim V\geq 3$.}
\vskip0.3cm\par\noindent

\section{Semisimplicity and finite-dimensionality}
\label{secfin}
\subsection{Semisimplicity of the maximal prolongation}\hfill\vskip0.3cm\par

Recall that a (possibly infinite-dimensional) Lie superalgebra is usually said to be \emph{semisimple} if its radical is zero, see \cite{Cheng1995, Kac06}. Equivalently, a Lie superalgebra is semisimple if it does not contain any non-zero abelian ideal.
The main goal of this section is to prove the following theorem.
\begin{theorem}
\label{nonloso}
\label{simpleprol}
Let $\dim V\geq 3$ and $\g=\bigoplus_{p\in\Z}\g_p$ be the maximal transitive prolongation of a supertranslation algebra $\m=V\oplus W$. Then exactly one of the following two cases occurs:
\begin{enumerate}
\item $\g_{p}=0$ for all $p\geq 1$;
\item $\g$ is semisimple and contains a unique minimal non-zero ideal $\s$.
\end{enumerate}
In the latter case, $\s$ is a simple transitive prolongation of $\m$ which contains $\bigoplus_{p\geq 0}\g_{2p+1}$ and the ideal $\so(V)$ of $\g_{0}$ described in Theorem \ref{lo0}.
\end{theorem}
\begin{remark}
Theorem \ref{nonloso} is also valid in the Lie algebra case, with essentially the same proof. It is then easy to show that $\g=\s$ is a (finite-dimensional) simple Lie algebra in case (2). This improves \cite[Theorem 2.7]{AlSa}, extending the classification results obtained in \cite{AlSa} to arbitrary $N\geq 1$. 
\end{remark}
The proof of Theorem \ref{nonloso} requires an intermediate result.
\begin{proposition}
\label{nonlosoforte}
Let $\g$ be the maximal transitive prolongation of a supertranslation algebra $\m=V\oplus W$ with $\dim V\geq 3$. If $\g_1\neq 0$, then any transitive prolongation $\mathfrak{q}$ of $\m$ such that
\begin{enumerate}
\item $\mathfrak{q}_1\neq 0$,
\item $\mathfrak{q}_0$ contains the ideal $\so(V)\oplus\mathbb{C}E$ of $\g_{0}$ described in Theorem \ref{lo0},
\item $\mathfrak{q}$ is preserved by any $0$-degree Lie superalgebra automorphism of $\g$,
\end{enumerate} 
is a semisimple Lie superalgebra.
\end{proposition}
\begin{proof}
Let $\mathfrak{k}$ be a non-zero ideal of $\mathfrak{q}$. Then $\mathfrak{k}$ is $\mathbb{Z}$-graded since $\mathfrak{q}$ contains the grading element $E$.
If $\mathfrak{k}_{-2}=0$ then, by non-degeneracy of $\mcB$, one gets $\mathfrak{k}_{-1}=0$ and then, by transitivity, $\mathfrak{k}_p=0$ for every $p\geq 0$.
It follows that every non-zero ideal of $\mathfrak{q}$ contains $\mathfrak{q}_{-2}=V$, since $\so(V)\subset \mathfrak{q}$ acts irreducibly on $V$. 

Hence, there exists a unique minimal ideal of $\mathfrak{q}$, we denote it by $\s$. It is $\Z$-graded, $\s_{-2}=V$, and $\s_{-1}\neq 0$ because $\s_{-1}\supset[\mathfrak{q}_1,V]$ and $[\mathfrak{q}_1,V]\neq 0$ by hypothesis (1) and Proposition \ref{finedopo}. 

The minimal ideal $\s$ is preserved by any automorphism of $\mathfrak{q}$ and, by hypothesis (3), also by any $0$-degree automorphism of the maximal prolongation. In particular, $\s_{-1}$ is preserved by all homomorphisms $\{\psi_v\}_{v\in V}$ described in Proposition \ref{hanome} and is thus a $\Cl(V)$-submodule of $W$.
\vskip0.3cm\par

Assume by contradiction that $\s$ is abelian.
It follows that $\s_{-1}$ is a non-zero $\mcB$-isotropic $\Cl(V)$-submodule of $W$ and that 
$$\s_{-1}^{\perp}=\{s\in W | \mcB(s,\s_{-1})=0\}=\{s\in W | [s,\s_{-1}]=0\} $$
is a proper $\Cl(V)$-submodule of $W$ containing $\s_{-1}$.

Denote by $\mathfrak{a}$ a $\Cl(V)$-submodule of $W$ which is 
complementary to $\s_{-1}^{\perp}$. As the bilinear form
$
\eta=\mcB|_{\mathfrak{s}_{-1}\otimes \mathfrak{a}}
$
is non-degenerate, one has the following decomposition of $\mathcal{C}l(V)$-modules:
\begin{equation}
\label{dede}
W=\mathfrak{a}\oplus\mathfrak{b}\oplus \s_{-1}\,,
\end{equation}
where 
$$\mathfrak{b}=\{s\in W | \mcB(s,\mathfrak{a}\oplus\s_{-1})=0\}=\{s\in W | [s,\mathfrak{a}\oplus\s_{-1}]=0\} $$ 
denotes the $\mcB$-orthogonal complement to $\mathfrak{a}\oplus \s_{-1}$ in $W$. 

The bilinear form $\mcB$ can be written in block-matrix form w.r.t. the decomposition \eqref{dede} as
\begin{equation}
\label{pezzii}
\mcB=\begin{pmatrix}
\tilde{\eta} & 0 & \epsilon\,{}^T\!\eta \\
0 & \hat{\eta} & 0 \\
\eta & 0 & 0 
\end{pmatrix}\;,\qquad 
\begin{aligned}
\eta&=\mcB|_{\mathfrak{s}_{-1}\otimes \mathfrak{a}}, \\
\tilde{\eta}&=\mcB|_{\mathfrak{a}\otimes\mathfrak{a}}, \\
\hat{\eta}&=\mcB|_{\mathfrak{b}\otimes\mathfrak{b}}.
\end{aligned}
\end{equation}
The key point now is to show that there always exists an appropriate choice of the $\Cl(V)$-submodule $\mathfrak{a}$ in such a way that $\tilde{\eta}=0$, that is $\mathfrak{a}$ is abelian.
\vskip0.3cm\par
Denote by 
$$
\eta^*:\s_{-1}\rightarrow \mathfrak{a}^*\,,\;\;\tilde{\eta}^*:\mathfrak{a}\rightarrow \mathfrak{a}^*\,,\;\; \hat{\eta}^*:\mathfrak{b}\rightarrow \mathfrak{b}^*
$$
the linear maps induced by \eqref{pezzii} and by $\psi_v^*:\g^*\rightarrow \g^*$ the dual map of the homomorphism
$\psi_v:\g\rightarrow \g$, for any $v\in V$. It is immediate to check that
$$\eta^*\circ \psi_v=\epsilon \psi_v^*\circ\eta^*\,,\;\;\tilde{\eta}^*\circ \psi_v=\epsilon \psi_v^*\circ\tilde{\eta}^*\,,\;\; \hat{\eta}^*\circ \psi_v=\epsilon \psi_v^*\circ\hat{\eta}^*$$
for any $v\in V$. This implies that the map
$$\varphi=(\eta^*)^{-1}\circ\tilde{\eta}^*:\mathfrak{a}\rightarrow\s_{-1}$$ 
is $\Cl(V)$-equivariant and that the $\Cl(V)$-submodule $\{2a-\varphi(a)|a\in\mathfrak{a}\}$ is $\mcB$-isotropic and complementary to $\s_{-1}^\perp$. 

Without loss of generality, one may hence assume that $\tilde{\eta}=0$.
\vskip0.3cm\par
To get a contradiction it is now sufficient to exhibit a $0$-degree automorphism $\chi:\g\rightarrow\g$ which satisfies $\chi(\s_{-1})=\mathfrak{a}$. Invariance of $\s$ under any $0$-degree automorphism of $\g$ gives the required contradiction.

Since any $0$-degree automorphism of a fundamental Lie superalgebra $\m$ can be canonically prolonged to an automorphism of its maximal transitive prolongation, it is sufficient in our case to define $\chi$ on $\m=V\oplus W$.

To this aim, fix an admissible bilinear form $\Phi$ on the $\Cl(V)$-module $\mathfrak{a}$ with invariants $(\tau',\sigma')$ and let $\mu=\tau'\sigma'$. Recalling that $[\mathfrak{s}_{-1},\s_{-1}]=[\mathfrak{a},\mathfrak{a}]=0$, it is not difficult to check that 
$$
\chi|_{V}=\mu\, \Id_{V}\,,\;\;\chi|_{\mathfrak{a}}=(\eta^*)^{-1}\circ\Phi^*\,,\;\;\chi|_{\mathfrak{b}}=\sqrt{\mu}\, \Id_{\mathfrak{b}}\,,\;\;\chi|_{\s_{-1}}=(\Phi^*)^{-1}\circ\eta^*\,,
$$
defines the required automorphism of $\m=V\oplus W$.
\end{proof}
\begin{proof}[\protect{Proof of Theorem \ref{nonloso}}]
Assume that $\g_1\neq0$. By Proposition \ref{nonlosoforte}, the maximal transitive prolongation $\g$ is a semisimple Lie superalgebra. 
Arguing as in the first part of the proof of Proposition \ref{nonlosoforte}, one can prove that every ideal of $\g$ is $\Z$-graded and contains $\g_{-2}$. 

The unique minimal non-zero ideal $\s$ of $\g$ is exactly the ideal generated by $\g_{-2}$ and it has a non-zero component $\s_{-1}$ in degree $-1$.
\vskip0.3cm\par

\paragraph{\it Claim I}
$\s_0\supset\so(V)$.
\vskip0.2cm\par
If $\s_0$ does not contain $\so(V)$, then, by $\so(V)$-invariance (for $\dim V\neq 4$; by invariance under $\so(V)$ and all automorphisms $\{\psi_v\}_{v\in V}$ if $\dim V=4$), one gets $\s_0\subset\C E\oplus\h_0$. 

Consider an element $D\in\s_1$. Then, for all $s\in\g_{-1}$, the element $Ds$ belongs to $\C E\oplus\h_0$ and,
for any $v,u\in\g_{-2}$, one gets
\begin{equation}\label{tralallala}
\begin{aligned}
0&=([Ds,v],u)-([Ds,u],v)=([s,Dv],u)-([s,Du],v)\\
  &=\mcB\big(u\cdot s,v\cdot\phi(D)\big)-\mcB\big(v\cdot s,u\cdot\phi(D)\big)\\
  &=\epsilon\mcB\big((vu-uv)\cdot s,\phi(D)\big),
\end{aligned}
\end{equation}
where $\phi: \g_{1}\to W$ is the embedding described in Lemma \ref{lemmag1}. 

It follows that $\phi(D)=0$. Therefore $\s_1=0$ and, by transitivity, $\s_p=0$ for all $p\geq 1$.
\vskip0.3cm
As $\s$ is minimal and not abelian (by semisimplicity of $\g$), one has $[\s,\s]=\s$. In particular, $\s_0=[\s_0,\s_0]$ is contained in $[\C E\oplus\h_0,\C E\oplus\h_0]\subset\h_0$ and the center $Z(\s)$ of $\s$ is an ideal in $\g$ containing $\g_{-2}$ and strictly contained in $\s$. This contradicts the minimality of $\s$ and proves that $\s_0\supset\so(V)$.
\vskip0.3cm\par
\paragraph{\it Claim II}
{\it $\s_p=\g_p$ for $p=-2,2$ and for all odd $p$. Moreover, $\s=[\g_{\bar1},\g_{\bar1}]\oplus\g_{\bar 1}$.}
\vskip0.2cm\par
First note that $\s$ contains all non-trivial irreducible $\so(V)$-submodules in $\g$, because $\s$ is an ideal containing $\so(V)$. For every $p\geq0$, one can identify $\g_p$ with an $\so(V)$-submodule of $W\otimes (W^*)^{\otimes p+1}\simeq W^{\otimes p+2}$, hence $\g_p$ is a direct sum of half-spin representations whenever $p$ is odd. It follows then that $\g_{p}\subset\s$ for every odd $p$. 

It remains to prove that $\g_{2}\subset\s$. Proposition~\ref{finedopo} implies that any trivial irreducible $\so(V)$-submodule of $\g_{2}$  
can be identified with  a one-dimensional subspace $\C D$  
of 
\[ \mathrm{Hom}(V,\g_0)^{\so(V)}\simeq\mathrm{Hom}\big(V,\Lambda^2(V)\oplus\C^r\big)^{\so(V)}=\mathrm{Hom}(V,\Lambda^2(V))^{\so(V)}\,, \]
where the first identification is induced by the isomorphism of $\so(V)$-modules
$$
\g_0\simeq\Lambda^2(V)\oplus\C^r\,,\,\,\,r=\dim(\C E\oplus\h_{0})\ .
$$
The subspace $\g_{-2}\oplus\so(V)\oplus\C D$ of $\g_{\bar 0}$ is then a Lie subalgebra of $\g$ and $D$ an element of the first term $\so(V)^{(1)}$ of the Cartan prolongation of $\so(V)$. It is well-known that $\so(V)^{(1)}=0$; it follows that $\g_2$ is a direct sum of non-trivial irreducible $\so(V)$-submodules and that $\g_{2}\subset\s$.

Finally, $[\g_{\bar1},\g_{\bar1}]\oplus\g_{\bar1}$ is an ideal of $\g$ contained in $\s$, and so it is equal to $\s$.
\vskip0.3cm\par

\paragraph{\it Claim III}
{\it $\s$ is simple.}
\vskip0.2cm\par
Let $\mathfrak{k}$ be a non-zero ideal of $\s$ and $X$ a non-zero element of $\k$. Then 
$$X=\sum_{i\in\Z} X_i\,\,,\quad X_i\in\s_i\,\,,$$ 
\vskip-0.15cm\par\noindent
is a finite sum of homogeneous elements; denote by $j$ the highest integer for which $X_{j}\neq0$. By transitivity and non-degeneracy of the maximal prolongation $\g$, there exist elements $s_1,\ldots,s_{j+2}\in\g_{-1}=\s_{-1}$ such that 
\[
0\neq[s_1,\cdots,[s_{j+2},X]\cdots]=[s_1,\cdots,[s_{j+2},X_{j}]\cdots]\in\k\cap\g_{-2}.
\]
It follows that $\g_{-2}$ is contained in every non-zero ideal $\k$ of $\s$. 

Denote then by $\k$ the ideal of $\s$ which is generated by $\g_{-2}$; it is the minimal non-zero ideal of $\s$ and it is $\Z$-graded.
By Proposition~\ref{finedopo} and transitivity, one has $\mathfrak{k}_{-1}\neq0$.
\vskip0.3cm\par

If $\mathfrak{k}_0$ has a non-trivial $\so(V)$ component, then, 
arguing as at the beginning of the proof of Claim I, $\mathfrak{k}_0\supset \so(V)$. Hence, proceeding as in Claim II, one gets
$$\mathfrak{k}_{\bar1}=\s_{\bar1}=\g_{\bar1}\,\,\,\text{and}\,\,\,\mathfrak{k}=[\g_{\bar1},\g_{\bar1}]\oplus\g_{\bar1}=\s\,.$$
It follows that $\s$ is simple in this case. 

On the other hand, the case $\mathfrak{k}_0\subset\C E\oplus\h_0$ can not happen. Indeed, arguing as in \eqref{tralallala}, one gets $\mathfrak{k}_p=0$ for all $p\geq 1$ and one finds a non-zero $\Z$-graded abelian ideal of $\s$, and hence of $\s+\C E$ (if $[\mathfrak{k},\mathfrak{k}]\neq0$, then $\mathfrak{k}_0\subset\h_0$ and the center $Z(\mathfrak{k})$ of $\mathfrak{k}$ is a non-zero abelian ideal of $\s$; if $[\mathfrak{k},\mathfrak{k}]=0$, then $\mathfrak{k}$ is abelian).
This gives a contradiction, since $\s+\C E$ is semisimple by Proposition \ref{nonlosoforte}.
\end{proof}

\subsection{Finite-dimensionality of the maximal prolongation}
\hfill\vskip0.3cm\par
The main aim of this section is to prove prove the following.

\begin{theorem}\label{findim}
The maximal transitive prolongation $\g$ of a supertranslation algebra $\m=V\oplus W$ with $\dim V\geq 3$ is finite-dimensional.
\end{theorem}

Before proving Theorem \ref{findim}, we briefly recall some important notions about infinite-dimensional filtered Lie superalgebras.

A \emph{filtration} of a Lie superalgebra $\L$ is a chain $\{\L_{(p)}\}_{p\in\Z}$ of linear subspaces 
$$\L_{(p)}\subset \L\,\,\,,\,\,\,\,\,\,\L_{(p)}=\L_{(p)}\cap\L_{\bar 0}\oplus\L_{(p)}\cap\L_{\bar 1}\,\,\,,$$ 
which satisfies
\[\begin{gathered}
\L=\bigcup_{p\in\mathbb Z}\L_{(p)},\\
0=\bigcap_{p\in \Z}\L_{(p)},\\
\end{gathered}
\]
and
\[\begin{gathered}
\L_{(p)}\supset\L_{(p+1)},\\
[\L_{(p)},\L_{(q)}]\subset\L_{(p+q)},\\
\end{gathered}
\]
for all $p,q\in\mathbb Z$. 

The \emph{depth} of the filtration $\{\L_{(p)}\}_{p\in\Z}$ is the least $d\in\mathbb Z$ such that $\L_{(-d)}=\L$, or $\infty$ if no such integer exists. A filtration is called \emph{regular} if it has finite depth and all the quotients $\L_{(p)}/\L_{(p+1)}$ are finite dimensional. 
We assume without further mention that all the filtrations that we consider are regular.

On a (regular) filtered Lie superalgebra $\L$ we consider the linear topology for which the filtration subspaces are a fundamental system of neighborhoods of $0$. Note that a subspace is open if and only if it is closed and of finite codimension.
If $\L$ is complete with respect to this topology, then it is \emph{linearly compact} in the sense of \cite{Gui2,Kac98}.

Given a filtered Lie superalgebra $\L$, one can consider the \emph{associated $\mathbb Z$-graded Lie superalgebra} $\mathfrak L=\gr(\L)$ which is defined by
\[\begin{aligned}
\mathfrak L&=\bigoplus_{p\in\mathbb Z}\mathfrak L_p\,,\,\,\,\,
\mathfrak L_p=\L_{(p)}/\L_{(p+1)}\,,
\end{aligned}\]
with the induced Lie bracket and parity decomposition.

Vice-versa, every $\Z$-graded Lie superalgebra $\g=\bigoplus_{p\geq-d}\g_p$ has a natural filtration:
\[ \g=\g_{(-d)}\supset\g_{(-d+1)}\supset\cdots\supset\g_{(p)}\supset\cdots \]
where $\g_{(p)}=\bigoplus_{i\geq p}\g_i$. 
The direct product $\bar\g=\prod_{p\geq -d}\g_p$ is also filtered by 
$$\big\{\bar\g_{(p)}=\prod_{i\geq p}\g_i\big\}_{p\in\Z}$$ 
and it is a complete topological filtered Lie superalgebra, hence linearly compact. There are a natural dense inclusion $\g\subset\bar\g$ and a natural isomorphism $\g\simeq\gr(\bar\g)$. 
\vskip0.3cm\par
Given a linearly compact Lie superalgebra $\L$, a proper open subalgebra $\L_0\subset \L$  that does not contain any non-zero ideal of $\L$ is called \emph{filtered-fundamental} (note that the notion of filtered-fundamental subalgebra of a filtered Lie superalgebra is not directly related to the notion of fundamental graded Lie superalgebra defined in Section \ref{Introduction}). 
The linearly compact Lie superalgebra $\L$ admits a filtered-fundamental subalgebra if and only if it satisfies an artinian condition: every descending sequence of closed ideals of $\L$ is eventually stable, see \cite{Gui2,Kac06}.
A maximal subalgebra $\L_0\subset\L$ that is also filtered-fundamental is called \emph{primitive}.

An even element $X\in\L$ is called \emph{exponentiable} if $\ad_{X}:\L\rightarrow\L$ leaves invariant any closed subspace $H\subset\L$ that is invariant under every continuous automorphism of $\L$
\cite{Gui2,CantKac}.

It is known that every even element of a filtered-fundamental subalgebra is exponentiable \cite{Gui2,Kac98}. This fact leads to a stronger notion of primitivity:
a primitive subalgebra $\L_0\subset\L$ which contains {\it all} exponentiable elements of $\L$ is called \emph{even primitive}.
\vskip0.3cm\par
The following intermediate result will be used in the proof of Theorem \ref{findim}.
\begin{proposition}\label{ep}
Let $\s$ be a simple  transitive prolongation  of a supertranslation algebra $\m=V\oplus W$ with $\dim V\geq 3$. If
\begin{enumerate}
\item[--] $\s_0$ contains the ideal $\so(V)$ of $\g_{0}$ described in Theorem  \eqref{lo0},
\item[--] $\s$ is invariant under all $0$-degree Lie superalgebra automorphisms of the maximal transitive prolongation $\g$ of $\m$,
\end{enumerate}
then the completion $\bar\s$ 
\begin{enumerate}
\item
is simple (i.e., it does not contain any proper non-zero closed ideal),
\item
admits a primitive subalgebra $\L_0$ which contains $\bar\s_{(0)}$.
\end{enumerate}
If $\s$ is infinite-dimensional, any primitive subalgebra $\L_0$ containing $\bar\s_{(0)}$ is even primitive and satisfies
$\bar\s_{(0)}\subset \L_0\subsetneq\bar\s_{(-1)}$.
\end{proposition}
\begin{proof}
Let $\k$ be a proper non-zero closed ideal of $\bar\s$. The associated $\Z$-graded subalgebra $\gr(\k)$ of $\gr(\bar\s)\simeq\s$ is a proper non-zero ideal of $\s$, proving (1).

Any maximal proper subalgebra $\L_0$ of $\bar\s$ containing $\bar\s_{(0)}$ is open and filtered-fundamental, proving (2).
\vskip0.3cm\par
Assume now that $\bar\s$ is infinite-dimensional.
Note first that $\L_0$ is $\Z$-graded. Indeed $$(\L_0)_{\overline 0}=(\bar\s_{(0)})_{\overline 0}\oplus(\L_0\cap \s_{-2})\;,\quad
(\L_0)_{\overline 1}=(\bar\s_{(0)})_{\overline 1}\oplus(\L_0\cap \s_{-1})\,,$$
since the $\Z$-grading of $\bar\s/\bar\s_{(0)}\simeq\m$ is consistent and of depth $2$.

If $\L_0\cap\s_{-2}\neq0$, then $\s_{-2}\subset \L_0$ and hence $\bar\s_{\bar 0}\subset \L_0$. However a linearly compact Lie superalgebra $\bar\s$ with a filtered-fundamental subalgebra $\L_0$ containing all even elements is necessarily finite-dimensional, as $\bar\s$ is isomorphic to a subalgebra of $\der\Lambda((\bar\s/\L_0)^*)$, by the superalgebra version \cite{Sche} of the Realization Theorem of Guillemin and Sternberg. One concludes that $\bar\s_{(0)}\subset \L_0\subsetneq\bar\s_{(-1)}$, the second inclusion being strict as
$\bar\s_{(-1)}$ is not a subalgebra.

One proves now that $\L_0$ is even primitive.
All the even elements of the filtered-fundamental subalgebra $\L_0$ are exponentiable. Assuming by contradiction that $\L_0$ is not even primitive, there exists a non-zero exponentiable element in $\s_{-2}=V$. 

The subspace of exponentiable elements in $V$ is invariant under all 0-degree Lie superalgebra automorphisms of the maximal transitive prolongation $\g$, and, in particular, under the restrictions to $V$ of  the automorphisms $\{\psi_v\}_{v\in V}$. 
Since these latter generate the orthogonal group $\mathrm{O}(V)$, all elements of $V$, and thus of $\bar\s_{\bar0}$, are exponentiable. 

On the other hand, there always exists a filtered-fundamental subalgebra containing all exponentiable elements \cite{Kac98}. The Realization Theorem would then imply that $\bar\s$ is finite-dimensional, which is an absurd.
\end{proof}

\begin{proof}[Proof of Theorem \ref{findim}]
Assume, by contradiction, that the maximal transitive prolongation $\g$ of $\m=V\oplus W$ is infinite-dimensional. Let $\s$ be the minimal ideal of $\g$ described in Theorem \ref{nonloso}. Note that $\s$ is an infinite-dimensional simple transitive prolongation of $\m$ and denote by $\bar\s$ its completion. 

Consider an even primitive subalgebra $\L_0$ with $\bar\s_{(0)}\subset \L_0\subsetneq\bar\s_{(-1)}$, as established in Proposition \ref{ep}, and a minimal $\L_0$-invariant subspace $\L_{(-1)}$ strictly containing $\L_0$.
By setting
\begin{align*}
\L_{(p)}&=\L_{(p+1)}+[\L_{(-1)},\L_{(p+1)}], &&\text{for every $p\leq-2$,}\\
\L_{(0)}&=\L_0\,\,\text{and}\\
\L_{(p)}&=\{X\in\L_{(p-1)}\mid[X,\L_{(-1)}]\subset \L_{(p-1)}\}, &&\text{for every $p\geq1$,}
\end{align*}
one obtains a filtration of $\bar\s$ 
\[ \bar\s=\L_{(-d)}\supset\cdots\supset \L_{(-1)}\supset \L_{(0)}\supset \L_{(1)}\cdots\,, \]
of depth $d\geq 1$, usually referred to as the {\it Weisfeiler filtration}, see \cite{Weis}. One has to distinguish two cases.
\vskip0.4cm\par

\paragraph{\it Case $\L_{(0)}=\bar\s_{(0)}$} \hfill
\vskip0.2cm\par
Let $S\subset\s_{-1}$ be a non-zero irreducible $\s_0$-submodule of $\s_{-1}$, and set 
$$\L_{(-1)}=\L_{(0)}\oplus S\ .$$ Since $\L_{(0)}$ is a maximal subalgebra, the subspace $\L_{(-1)}$ generates $\bar\s$.

One has $[\L_{(-1)},\L_{(-1)}]\subset \L_{(-1)}\oplus[S,S]$ and, since $\L_{(-1)}$ is not a subalgebra, one obtains $[S,S]=V$. On the other hand
\[ [\L_{(-1)},[\L_{(-1)},\L_{(-1)}]]\subset[\L_{(-1)},\L_{(-1)}]+[\L_{(-1)},V]\subset (\bar\s_{(0)}\oplus S \oplus V) + [\s_1,[S,S]] \]
and $[\s_1,[S,S]]\subset [[\s_1,S],S]\subset S$ imply that $$\L_{(-1)}+[\L_{(-1)},\L_{(-1)}]$$ is a subalgebra. It follows that $$\L_{(-1)}+[\L_{(-1)},\L_{(-1)}]=\bar\s,$$ $W$ is $\s_0$-irreducible and equal to $S$, and $\L_{(-1)}=\bar\s_{(-1)}$.

The Weisfeiler filtration is then given in this case by $\L_{(p)}=\bar\s_{(p)}$ for every $p\in\Z$ (in particular, the depth $d$ is equal to $2$).
\vskip0.4cm\par

\paragraph{\it Case $\L_{(0)}\supsetneq\bar\s_{(0)}$}\hfill
\vskip0.2cm\par
Let   $\L_{(-1)}$ be a minimal $\L_{(0)}$-invariant subspace strictly containing $\L_{(0)}$, and set $S=\L_{(-1)}\cap\s_{-1}$. Note that $\L_{(0)}\cap\s_{-1}$ is an abelian subspace of $\s_{-1}$. 

With an argument similar to the previous case, one gets  $S=\s_{-1}$ and $\L_{(-1)}=\bar\s$. Hence, the depth $d$ is equal to $1$ and the first few terms of the Weisfeiler filtration are given by
$$
\bar\s=\L_{(-1)}\supset \L_{(0)}\supset \L_{(1)}\supset\cdots\,\,,
$$
where
\begin{equation*}
\L_{(1)}=\bar\s_{(2)}\oplus\{D\in\s_{1}\mid [D,\s_{-2}]\subset \L_{(0)}\cap\s_{-1}\}\oplus\{D\in\h_0\mid [D,\s_{-1}]\subset \L_{(0)}\}.
\end{equation*}
\vskip0.4cm\par
Let now $\mathfrak L=\bigoplus_{p\geq -d}\mathfrak L_p$ be the $\Z$-graded Lie superalgebra associated to the Weisfeiler filtration $\{\L_{(p)}\}_{p\in\Z}$. In the terminology of \cite{Kac98}, $\gL$ is a \emph{$\Z$-graded even transitive irreducible infinite-dimensional} Lie superalgebra.
Moreover, again by \cite{Kac98}, the $\gL_0$-module $\mathfrak L_{-1}$ is \emph{strongly transitive}, i.e., it is a faithful irreducible $\gL_0$-module such that $[\gL_0,X]=\gL_{-1}$ for all non-zero even elements $X\in\gL_{-1}$. 

To conclude the proof of the theorem, one uses the classification \cite{Kac98} of such Lie superalgebras and modules, distinguishing again the two cases above.
\vskip0.4cm\par

\paragraph{\it Case $\L_{(0)}=\bar\s_{(0)}$}\hfill
\vskip0.2cm\par
The $\Z$-graded Lie superalgebra $\gL$ coincides with $\s$. Thus the $\Z$-grading is consistent and the non-positive part of $\gL$ satisfies
\begin{enumerate}
\item[$(1)$]
$\gL_{-p}=0$ for every $p\geq 3$;
\item[$(2)$]
$\dim \gL_{-2}=\dim V\geq 3$;
\item[$(3)$]
$\gL_0$ contains an ideal isomorphic to $\so(V)$ acting on $\gL_{-2}= V$ via the tautological representation.
\end{enumerate}
The even transitive irreducible infinite-dimensional consistently $\Z$-graded Lie superalgebras of depth $d\geq 2$ are listed in \cite[Thm.~5.3]{Kac98} and described in detail in \cite{CK99}.

A case by case verification shows that an $\gL$ satisfying the above properties $(1)-(3)$ does not exist.
\vskip0.4cm\par

\paragraph{\it Case $\L_{(0)}\supsetneq\bar\s_{(0)}$}\hfill
\vskip0.2cm\par
The $\Z$-grading of $\gL$ is not consistent and the non-positive part of $\gL$ satisfies
\begin{enumerate}
\item[$(1)$]
$\gL_{-p}=0$ for every $p\geq 2$;
\item[$(2)$]
$(\gL_{-1})_{\bar0}$ has dimension greater or equal than $3$;
\item[$(3)$]
$(\gL_0)_{\bar0}$ contains an ideal isomorphic to $\so(V)$ acting on $(\gL_{-1})_{\bar0}= V$ via the tautological representation.
\end{enumerate}
The strongly transitive modules with a non-zero even component are listed in \cite[Thm~3.1]{Kac98} (and corrected in \cite{CKadd}).

A case by case verification shows again that an $\gL$ satisfying the above properties $(1)-(3)$ does not exist.
\end{proof}

\section{Classification of simple and maximal prolongations.}
\label{sec2}
In Section \ref{secfin}, we proved that the maximal transitive prolongation $$\g=\bigoplus_{p\geq -2}\g_{p}$$ of a supertranslation algebra $\m=V\oplus W$ with $\dim V\geq 3$ is a finite-dimensional Lie superalgebra (Theorem \ref{findim}) which is semisimple if $\g_{1}\neq 0$ (Theorem \ref{nonloso}). In the latter case, $\g$ contains a unique minimal ideal $\s$ which is a simple prolongation of $\m$ (Theorem \ref{nonloso}). 

This section contains our main classification results: we first classify all possible simple prolongations of $\m$ (Theorem \ref{classimple}) and then derive the maximal transitive prolongations containing each of them (Theorem \ref{teorematotale}). 
\begin{remark}
In the Lie algebra case, if $\m$ is a negatively graded fundamental Lie algebra whose maximal transitive prolongation $\g$ is finite-dimensional, 
any prolongation $\s$ of $\m$ that is simple necessarily coincides with $\g$ \cite{MN}. The corresponding statement is not true in the Lie superalgebra case
and one has to separately consider simple and maximal prolongations.
\end{remark}

Finite-dimensional simple Lie superalgebras $\g=\g_{\bar0}\oplus\g_{\bar1}$ are classified, see \cite{Kac77a} and references therein, and split into two main families:
{\it classical} Lie superalgebras, for which the adjoint action of $\g_{\bar 0}$ on $\g_{\bar1}$ is completely reducible, and
{\it Cartan} Lie superalgebras $W(n)$ (for $n\geq 3$), $S(n)$ (for $n\geq 4$), $\tilde S(n)$ (for $n\geq 4$ and even), $H(n)$ (for $n\geq 5$), that is finite-dimensional Lie superalgebras analogue to simple Lie algebras of vector fields.
\begin{remark}
The simple Lie superalgebras $W(2)$, $S(3)$, $\tilde{S}(2)$ and $H(4)$ are isomorphic to the classical Lie superalgebras $\mathfrak{sl}(1|2)\simeq \mathfrak{osp}(2|2)$, $\mathfrak{spe}(3)$, $\mathfrak{osp}(1|2)$ and $\mathfrak{psl}(2|2)$ respectively. In our conventions, they are not Cartan Lie superalgebras.
\end{remark}
Classical Lie superalgebras in turn split into the so-called {\it basic} Lie superalgebras $\mathfrak{sl}(m+1|n+1)$ (for $m<n$), $\mathfrak{psl}(n+1|n+1)$ (for $n\geq 1$), $\mathfrak{osp}(2m+1|2n)$ (for $n\geq 1$), $\mathfrak{osp}(2|2n-2)$ (for $n\geq 2$),
$\mathfrak{osp}(2m|2n)$ (for $m\geq 2$, $n\geq 1$), $\mathfrak{osp}(4|2;\alpha)$ (for $\alpha\neq 0, \pm 1, \infty$), $\mathfrak{ab}(3)$, $\mathfrak{ag}(2)$,
for which there exists a non-degenerate even invariant supersymmetric bilinear form, and two {\it strange} families $\mathfrak{spe}(n)$ (for $n\geq 3$) and $\mathfrak{psq}(n)$ (for $n\geq 3$).

Note that some authors use different conventions to 
denote some of the above classical Lie superalgebras, cf. \cite{Kac77a, Sche}. 
Our conventions are consistent with those used in, e.g., \cite{Serga, Leites}. 
\vskip0.3cm\par
We first prove that the Cartan and strange Lie superalgebras do not appear in our classification and then deal with basic Lie superalgebras.

\subsection{Cartan type and strange Lie superalgebras}\hfill
\vskip0.3cm\par
We briefly recall the description of Cartan Lie superalgebras and their possible $\Z$-gradings \cite{Kac77b}.

Let 
$$W(n)=\{\sum_{\alpha=1}^{n}P^{\alpha}(\xi^1,\dots,\xi^n)\frac{\partial}{\partial \xi^\alpha}\mid P^{\alpha}\in\Lambda((\mathbb C^{n})^{*}),\,1\leq \alpha\leq n\}$$ 
be the algebra of derivations of the Grassmann algebra 
$\Lambda((\mathbb C^{n})^{*})$ generated by $n$ elements $\xi^1,\dots,\xi^n\in(\mathbb C^{n})^{*}$.
Given an $n$-tuple of integers $(k_1,\dots,k_n)$, the $\Z$-grading of $W(n)$ of type $(k_1,\dots,k_n)$ is defined by assigning degrees 
$$
\deg(\xi^\alpha)=k_\alpha\,,\;\;\deg(\frac{\partial}{\partial \xi^\alpha})=-k_\alpha\ .
$$
Up to isomorphism, every $\Z$-grading of $W(n)$ is of this form. 
Note that a grading of type $(k_1,\dots,k_n)$ is consistent precisely if all $k_\alpha$'s are odd.

The grading of type $(1,\dots,1)$ is usually called the {\it principal grading} and, in this case, the subalgebra 
$$W(n)_0=\langle \xi^\alpha\frac{\partial}{\partial \xi^\beta}\mid 1\leq \alpha,\beta\leq n \rangle$$ 
of $0$-degree
elements is identifiable with $\mathfrak{gl}(n)$. The Lie subalgebra 

\begin{equation}
\label{slmc}
\mathfrak{sl}(n)=\langle \xi^\alpha\frac{\partial}{\partial \xi^\beta}, \xi^\alpha\frac{\partial}{\partial \xi^\alpha}-\xi^\beta\frac{\partial}{\partial \xi^\beta} \mid 1\leq \alpha\neq \beta\leq n \rangle
\end{equation}
is a Levi factor of the even part $W(n)_{\bar{0}}$; it is $\Z$-graded in every grading of type $(k_1,\dots,k_n)$.

We denote by $S(n)$ the subalgebra of divergence free derivations of $\Lambda((\mathbb C^{n})^{*})$, where $\operatorname{div}(\sum_{\alpha} P^{\alpha}\frac{\partial}{\partial \xi^\alpha})=\sum_{\alpha}(-1)^{p(P^{\alpha})}\frac{\partial P^{\alpha}}{\partial \xi^{\alpha}}$, and by $$\tilde{S}(n)=(1+\xi^1\cdots\xi^n)S(n)$$ the unique non-trivial simple deformation of $S(n)$ for $n$ even (see \cite{Kac77a} for more details). 

Finally, the Hamiltonian Lie superalgebra $H(n)$ is the derived ideal of the superalgebra preserving the symplectic form
$\omega_n=\sum_{i=1}^{n}d\xi^{i}d\xi^{n+1-i}$, i.e. 
$$H(n)=\bigg\langle\sum_{\alpha=1}^{n}\frac{\partial f}{\partial \xi^\alpha} \frac{\partial}{\partial \xi^{n+1-\alpha}}+\frac{\partial f}{\partial \xi^{n+1-\alpha}}
\frac{\partial}{\partial \xi^\alpha}\mid f\in\Lambda^{k}((\mathbb C^{n})^{*})\,,1\leq k\leq n-1 \bigg\rangle\, .$$

Up to isomorphism, all $\Z$-gradings of the simple Lie superalgebras $S(n)$, $\tilde{S}(n)$ and $H(n)$ are induced by $\Z$-gradings of $W(n)$.
More precisely, they are all obtained as follows (see \cite{Kac77g}, where there is a misprint in the $H(n)$-case):
\begin{itemize}
\item[i)] every grading of type $(k_1,\dots,k_n)$ induces a $\Z$-grading of $S(n)$,
\item[ii)] every grading of type $(k_1,\dots,k_n)$ with $\sum_{i=1}^{n} k_i=0$ induces a $\Z$-grading of $\tilde{S}(n)$,
\item[iii)] every grading of type $(k_1,\dots,k_n)$ with $k_i+k_{n+1-i}=k_j+k_{n+1-j}$ for all $1\leq i,j\leq n$ induces a $\Z$-grading of $H(n)$.
\end{itemize}
A $\Z$-grading of $S(n)$, $\tilde{S}(n)$ or $H(n)$ is consistent if and only if it is induced by a consistent $\Z$-grading of $W(n)$.

The subalgebra \eqref{slmc} is also a Levi factor of $S(n)_{\bar{0}}$ and $\tilde{S}(n)_{\bar{0}}$.
On the other hand, the intersection $\mathfrak{sl}(n)\cap H(n)$ gives the Levi factor 
\begin{equation}
\label{somc}
\so(n)=\langle \xi^\alpha\frac{\partial}{\partial \xi^{n+1-\beta}}-\xi^{\beta}\frac{\partial}{\partial \xi^{n+1-\alpha}}\mid 1\leq \alpha<\beta\leq n \rangle
\end{equation}
of $H(n)_{\bar{0}}$ and it is $\Z$-graded in every grading of $H(n)$ of type $(k_1,\dots,k_n)$.

\begin{proposition}
Let $\m=V\oplus W$ be a supertranslation algebra satisfying $\dim V\geq 3$ and $\s=\bigoplus_{p\geq -2}\s_{p}$ a simple prolongation of $\m$. If
$\s_0$ contains an ideal isomorphic to $\so(V)$, acting via the tautological representation on $\m_{-2}= V$ and a multiple of the spinor representation on $\m_{-1}= W$, then $\s$ is not a Lie superalgebra of Cartan type.
\end{proposition}
\begin{proof}
Let $\s=W(n)$ (for $n\geq3$), $S(n)$ (for $n\geq4$), $\tilde{S}(n)$ (for $n\geq4$ and even) or $H(n)$ (for $n\geq 5$). The $\Z$-grading of $\s$ is induced by a grading of $W(n)$ of type $(k_1,\dots,k_n)$ and
there exists a $\Z$-graded Levi decomposition $$\s_{\bar{0}}=\l\oplus\r$$ of the even part of $\s$, where $\r$ denotes the radical of $\s_{\bar 0}$ and the simple Levi factor $\l$ is the one in \eqref{slmc} or \eqref{somc}. 

Since the $\Z$-grading of $\s$ has depth $2$, there are two possibilities for the $\Z$-grading of $\l$: either $\l=\l_0$ or $\l=\l_{-2}\oplus\l_0\oplus\l_2$ with $\l_{-2}\neq 0$. One has to treat the two cases separately.
\vskip0.4cm\par

\paragraph{\it Case $\l=\l_0$}\hfill
\vskip0.3cm\par
The subalgebra $\l$ is also a Levi factor of $\s_0$. Being simple, it coincides with $\so(V)$. This implies that $\s=W(4)$, $S(4)$, $\tilde{S}(4)$ and $\dim V=6$, or $\s=H(n)$ and $\dim V=n\geq 5$. The condition $\l\subset\s_0$ also implies that the $n$-tuple $(k_1,\dots,k_n)$ is an integer multiple of $(1,\dots,1)$. The $\Z$-grading of $\s$ is fundamental, forcing $(k_1,\dots,k_n)=\pm(1,\dots,1)$, and of depth $2$, leaving only the possibility 
$(k_1,\dots,k_n)=(-1,\dots,-1)$.  

This $\Z$-grading has depth $3$ for $W(4)$. For $S(4)$, it has depth $2$ but $\dim\s_{-2}=10\neq6$. For $\tilde{S}(4)$, the $\Z$-grading is not permissible. For $H(n)$, it has depth $n-3$ and, in the special case $H(5)$, the $\so(V)$-module $\s_{-1}$ is isomorphic to $\Lambda^3(V)$, which is not a multiple of the spinor module.
\vskip0.4cm\par

\paragraph{\it Case $\l=\l_{-2}\oplus\l_0\oplus\l_2$ with $\l_{-2}\neq 0$}\hfill\vskip0.3cm\par 
If $\l_{-2}\subsetneq V$, then $\r_{-2}\neq 0$ and this would imply $\r_{-2}=V$, a contradiction. Hence $\l_0$ acts irreducibly and conformally on $\l_{-2}=V$. Since  $\l_{2}\simeq V^*$,
one can see that $\l_0\simeq \mathfrak{co}(V)$ and $\l$ is isomorphic to $\so(D+2)$, where $D=\dim V$. 

This happens precisely when $\l=\mathfrak{sl}(4)\simeq\so(6)$ and $\s=W(4)$, $S(4)$, $\tilde{S}(4)$ or when $\l=\so(n)$ and $\s=H(n)$.

In the former case, the conditions that $\s$ has depth $2$ and $\dim V=4$ rule out all permissible $\Z$-gradings of $\s=W(4)$, $S(4)$, $\tilde{S}(4)$.

In the latter case, it is straightforward to check that $H(n)$ with $n\geq 5$ does not admit any consistent $\Z$-grading of depth $2$ with $H(n)_{2}\neq 0$.
\end{proof}

Having dealt with Lie superalgebras of Cartan type, we turn now to the two families of strange Lie superalgebras.

\begin{proposition}
Let $\m=V\oplus W$ be a supertranslation algebra satisfying $\dim V\geq 3$ and $\s=\bigoplus_{p\geq -2}\s_{p}$ a simple prolongation of $\m$. If
$\s_0$ contains an ideal isomorphic to $\so(V)$, acting via the tautological representation on $\m_{-2}= V$ and a multiple of the spinor representation on $\m_{-1}= W$, then $\s$ is not a strange Lie superalgebra
$\mathfrak{spe}(n)$ or $\mathfrak{psq}(n)$ (for $n\geq3$).
\end{proposition}
\begin{proof}
Assume $\s=\mathfrak{spe}(n)=\mathfrak{spe}(n)_{\bar0}\oplus \mathfrak{spe}(n)_{\bar1}$, where 
$$
\mathfrak{spe}(n)_{\bar0}\simeq\mathfrak{sl}(n)\,\,,\,\,\mathfrak{spe}(n)_{\bar1}\simeq \Pi(S^2(\C^n)\oplus\Lambda^2((\C^n)^*))\ .
$$ The even part is $\Z$-graded $\s_{\bar0}=\s_{-2}\oplus\s_0\oplus\s_2$ and $\s_{0}$ contains $\so(V)$ as an ideal acting via the tautological representation on $\s_{-2}=V$. The only possibility is that $n=4$ and $\mathfrak{sl}(4)\simeq\so(6)$ is the Cartan prolongation of $(\C^4,\mathfrak{co}(4))$. 

However the classification in \cite{Kac77g} implies that all consistent $\Z$-gradings of $\mathfrak{spe}(4)$ with $\mathfrak{spe}(4)_0\simeq \mathfrak{co}(4)$ have depth at least $3$.

Finally, $\mathfrak{psq}(n)$ does not admit any consistent $\Z$-grading.
\end{proof}
\vskip0.3cm\par
\subsection{Basic Lie superalgebras}\hfill\label{ebbeneebbene}
\vskip0.3cm\par

We briefly recall some notions about basic Lie superalgebras, their Dynkin diagrams and their $\mathbb{Z}$-gradings.

A simple Lie superalgebra $\s=\s_{\overline{0}}\oplus\s_{\overline{1}}$ is called {\it basic} if the even part $\s_{\overline{0}}$ is a reductive Lie algebra and there exists an even non-degenerate
invariant supersymmetric bilinear form $B:\s\otimes\s\rightarrow\C$. 

There are four families $\mathfrak{sl}(m+1|n+1)$ (for $m<n$) and $\mathfrak{psl}(n+1|n+1)$ (for $n\geq 1$), $\mathfrak{osp}(2m+1|2n)$ (for $n\geq 1$), $\mathfrak{osp}(2|2n-2)$ (for $n\geq 2$), $\mathfrak{osp}(2m|2n)$ (for $m\geq 2$, $n\geq 1$), a family $\mathfrak{osp}(4|2;\alpha)$ (for $\alpha\neq 0,\pm 1,\infty$) of deformations of $\mathfrak{osp}(4|2)$
and the exceptional cases $\mathfrak{ab}(3)$ and $\mathfrak{ag}(2)$; the list can be found in \cite{Kac77a}.

The form $B$ is unique up to constant and it coincides with the Killing form of $\s$, except for the cases $\mathfrak{psl}(n+1|n+1)$, $\mathfrak{osp}(2m+2|2m)$ and $\mathfrak{osp}(4|2;\alpha)$.

For later use, we give in Table \ref{Kactab} the description of the even Lie subalgebra $\s_{\overline{0}}$ of $\s$ and its representation on the odd part $\s_{\overline{1}}$. 
\vskip0.2cm
{\small 
\begin{table}[H]
\begin{centering}
\begin{tabular}{|c|c|c|}
\hline
$\s$ & $\s_{\overline{0}}$ & $\s_{\overline{1}}$\\
\hline
\hline
$\begin{gathered}\mathfrak{sl}(m+1|n+1)^{\phantom{T}}\\ m<n\end{gathered}$ & $\mathfrak{sl}(m+1)\oplus\mathfrak{sl}(n+1)\oplus\C Z$ & $\C^{m+1}\otimes(\C^{n+1})^{*}\oplus(\C^{m+1})^{*}\otimes\C^{n+1}$ \\
\hline
$\begin{gathered}\mathfrak{psl}(n+1|n+1)^{\phantom{T}}\\ n\geq 1 \end{gathered}$ & $\mathfrak{sl}(n+1)\oplus\mathfrak{sl}(n+1)$ & $\C^{n+1}\otimes(\C^{n+1})^{*}\oplus(\C^{n+1})^{*}\otimes\C^{n+1}$\\
\hline
$\begin{gathered}\mathfrak{osp}(2m+1|2n)^{\phantom{T}}\\ n\geq 1 \end{gathered}$ & $\mathfrak{so}(2m+1)\oplus\mathfrak{sp}(2n)$ & $\C^{2m+1}\otimes\C^{2n}$\\
\hline
$\begin{gathered}\mathfrak{osp}(2|2n-2)^{\phantom{T}}\\ n\geq 2 \end{gathered}$ & $\mathfrak{so}(2)\oplus\mathfrak{sp}(2n-2)$ & $\C^{2}\otimes\C^{2n-2}$\\
\hline
$\begin{gathered}\mathfrak{osp}(2m|2n)^{\phantom{T}}\\ m\geq 2, n\geq 1 \end{gathered}$ & $\mathfrak{so}(2m)\oplus\mathfrak{sp}(2n)$ & $\C^{2m}\otimes\C^{2n}$\\
\hline
$\begin{gathered}\mathfrak{osp}(4|2;\alpha)^{\phantom{T}}\\ \alpha\neq 0,\pm 1,\infty \end{gathered}$ & $\mathfrak{sl}(2)\oplus\mathfrak{sl}(2)\oplus\mathfrak{sl}(2)$ & $\C^{2}\otimes\C^{2}\otimes\C^2$\\
\hline
$\mathfrak{ab}(3)^{\phantom{T}}$ & $\mathfrak{so}(7)\oplus\mathfrak{sl}(2)$ & $\mathbb{S}\otimes\C^{2^{\phantom{T}}}$ \\
\hline
$\mathfrak{ag}(2)^{\phantom{T}}$ & $\mathrm{G}_2\oplus \mathfrak{sl}(2)$ & $\mathbb{C}^{7^{\phantom{T}}}\otimes\C^2$\\
\hline 
\end{tabular}
\end{centering}
\caption[]{\label{Kactab}}   
\vskip14pt
\end{table}}
\subsubsection*{Dynkin diagrams}\hfill\vskip0.2cm\par
Basic Lie superalgebras can be described by means of Cartan matrices, more precisely they are the quotients of indecomposable finite-dimensional {\it contragredient} Lie superalgebras by their center \cite{Kac77a, Kac77b}.
In all cases, the center is trivial with the exception of $\mathfrak{psl}(n+1|n+1)$, where the contragredient Lie superalgebra $\mathfrak{sl}(n+1|n+1)$ has a one-dimensional center. 

It is convenient to describe integer and sparse Cartan matrices by Dynkin diagrams. They were first introduced for Lie superalgebras in \cite{Kac77a, Serga}, however we will use the slightly different conventions given by \cite{Leites, BGL}. We recall here only the facts that we need and refer to those texts for more details.

Let $\g$ be an indecomposable finite-dimensional contragredient Lie superalgebra, $\t$ a Cartan subalgebra of $\g_{\bar0}$ and $\Delta=\Delta(\g,\t)$ the associated root system. Then $\g_{\bar0}$ and $\g_{\bar1}$ decompose into the direct sum of root spaces $\g^\alpha$ and a root $\alpha$ is called \emph{even} (resp. \emph{odd}) if $\g_{\bar0}^\alpha$ (resp. $\g_{\bar1}^\alpha$) is non-zero. Every root is either even or odd and the root spaces are one-dimensional except in the case $\g=\mathfrak{sl}(2|2)$, where all four odd roots have two-dimensional eigenspaces. 

Many properties of root systems of Lie algebras remain true for basic Lie superalgebras, see \cite[Proposition 5.3]{Kac77b}. In particular, any decomposition $\Delta=\Delta^+\cup-\Delta^+$ into positive and negative roots determines a system $$\Sigma=\{\alpha_1,\ldots,\alpha_r\}$$ of simple positive roots. Every positive root $\alpha\in\Delta^+$ can be written in a canonical way as a sum 
$$\alpha=\sum_{i=1}^{r}b_i\alpha_i,$$ with non-negative integer coefficients $b_i$.

The Weyl group of $\g_{\bar0}$ acts on the set of simple root systems. In contrast with the Lie algebra case this action is not transitive, and hence different simple root systems of the same 
basic Lie superalgebra may not be conjugated. 
To each orbit of the Weyl group one can associate a Dynkin diagram as follows.
\vskip0.2cm\par
Consider a Cartan matrix $(a_{ij})$ associated to $\Sigma$, see \cite{Leites, BGL}. Each simple root $\alpha_i$ corresponds to a node which is colored \emph{white} if $\alpha$ is even (in this case $a_{ii}=2)$, \emph{gray} if $\alpha$ is odd and $B$-isotropic (in this case $a_{ii}=0)$, or \emph{black} if $\alpha$ is odd and non-isotropic (in this case $a_{ii}=1$).

The $i$-th and $j$-th nodes of the diagram are not joined if $a_{ij}=a_{ji}=0$, otherwise they are joined by 
$\operatorname{max}(|a_{ij}|,|a_{ji}|)$-edges with an arrow pointing from $\alpha_i$ to $\alpha_j$ if $|a_{ij}|<|a_{ji}|$.




Finally, we mark the $i$-th node with the coefficient $b_{i,\text{max}}$ of the expression of the highest root 
$$\alpha_{\text{max}}=\sum_{i=1}^{r}b_{i,\text{max}}\alpha_i$$ as sum of simple roots.
 
A list of all possible Dynkin diagrams associated to basic Lie superalgebras is contained in \cite[Tables 1-5]{Leites}. For more details on how to recover the Cartan matrix from the Dynkin diagram, we refer the reader to \cite{Leites}.
\vskip0.2cm\par
\subsubsection*{Fundamental consistent $\Z$-gradings}\hfill
\vskip0.2cm\par
Let $\s$ be a finite-dimensional contragredient Lie superalgebra different from $\mathfrak{sl}(2|2)$, with a Cartan subalgebra $\t$, a root system $\Delta$ and a fixed simple root system $\Sigma$.

Let $\deg\alpha_i=0$ if $\alpha_i\in\Sigma$ is even and $\deg\alpha_i=1$ if $\alpha_i\in\Sigma$ is odd, and extend the definition to all roots by
\[\deg (\sum_{i=1}^{r} b_i\alpha_i)=\sum_{i=1}^{r} b_i\deg(\alpha_i). \]
The $\Z$-grading of $\g$ given by
$$
\g_0=\t\oplus\bigoplus_{\substack{\alpha\in\Delta\\\deg\alpha=0}}\g^\alpha\;,\qquad
\g_p=\bigoplus_{\substack{\alpha\in\Delta\\\deg\alpha=p}}\g^\alpha\;,
$$
is consistent and fundamental.

By \cite{Kac77g}, all possible consistent fundamental $\Z$-gradings of $\s$ are equivalent to one of this form, for some choice of $\Sigma$. In particular: 
\vskip0.3cm\par\noindent
{\it Every Dynkin diagram canonically describes a unique consistent and fundamental $\Z$-grading.}
\vskip0.3cm\par
The depth of $\s$ is equal to the degree of the highest root 
$$\deg(\alpha_{\text{max}})=\sum_{i=1}^{r}b_{i,\text{max}}\deg(\alpha_i)\ .$$
The subalgebra $\s_0$ is a reductive Lie algebra, the Dynkin diagram of its semisimple ideal is obtained from the Dynkin diagram of $\s$ by removing all gray and black nodes, and any line issuing from them.
\vskip0.2cm\par

\subsubsection*{Main classification result}\hfill\vskip0.2cm\par
We can now state and prove the following theorem.
\begin{theorem}
\label{classimple}
Let $\s$ be a basic Lie superalgebra satisfying the following properties:
\begin{enumerate}
\item $\s=\bigoplus_{p\geq -2}\s_{p}$ is a consistently $\Z$-graded Lie superalgebra of depth $2$ with $\dim\s_{-2}\geq 3$,
\item the $\Z$-grading is fundamental and transitive,
\item there exist an identification between $\s_{-2}$ and a vector space $V$ endowed with a non-degenerate symmetric bilinear form, and an ideal of $\s_{0}$
whose action on $\s_{-2}$ is equivalent to the tautological representation of $\so(V)$ on $V$,
\item the adjoint action of $\so(V)\subset \s_{0}$ on $\s_{-1}$ is equivalent to a direct sum of spinor or semi-spinor representations.
\end{enumerate}
Then $\s$ is isomorphic to one of the Lie superalgebras listed in Table \ref{primatab}, with the consistent $\Z$-grading determined by the corresponding Dynkin diagram. 
Therein, the symbol "$\cdots$" denotes a subdiagram corresponding to a Lie algebra $\mathfrak{sl}(\ell+1)$ of appropriate (possibly zero) rank $\ell$, whereas the symbol $E$ in the last column of the table is the grading element of $\mathfrak{s}$.
\end{theorem}
\begin{remark}
We recall that the simple Lie superalgebra $\mathfrak{psl}(n+1|n+1)$ (for $n\geq 1$) is the quotient of $\mathfrak{sl}(n+1|n+1)$ by its one-dimensional center.
Moreover, the derivations of $\mathfrak{sl}(n+1|n+1)$ and $\mathfrak{psl}(n+1|n+1)$ are all induced by the adjoint action of elements in $\mathfrak{gl}(n+1|n+1)$; in particular, the center of $\mathfrak{sl}(n+1|n+1)$ has degree zero for any $\mathbb{Z}$-grading of $\mathfrak{sl}(n+1|n+1)$ and  $\mathbb{Z}$-gradings of $\mathfrak{sl}(n+1|n+1)$ are in natural correspondence with those of $\mathfrak{psl}(n+1|n+1)$.

Hence the Dynkin diagram in the $\mathfrak{psl}(4|4)$-row of Table \ref{primatab} has to be understood as the Dynkin diagram of the contragredient Lie superalgebra $\mathfrak{sl}(4|4)$ and the consistent $\Z$-grading of $\mathfrak{psl}(4|4)$ as the one induced from $\mathfrak{sl}(4|4)$.
\end{remark}

\begin{sidewaystable}
\centering
\vskip13cm
\begin{tabular}{|c|c|c|c|c|}
\hline
$\s$ & Dynkin diagram & $\s_{-2}$ & $\s_{-1}$ & $\s_0$\\
\hline
\hline
&&&&\\[-3mm]
$\begin{gathered}\mathfrak{sl}(m+1|4)\\m\neq3\end{gathered}$&
\begin{tikzpicture}
\node[root]   (1)       [label=${{}^1}$]              {};
\node[xroot] (2) [right=of 1] {}[label=${{}^1}$] edge [-] (1);
\node[]   (3) [right=of 2] {$\;\cdots\,$} edge [-] (2);
\node[xroot]   (4) [right=of 3] {}[label=${{}^1}$] edge [-] (3);
\node[root]   (5) [right=of 4] {}[label=${{}^1}$] edge [-] (4);
\end{tikzpicture}&
$\C^ 4$ & $\mathbb{S}^+\otimes\C^{m+1}\oplus\mathbb{S}^-\otimes(\C^{m+1})^{*}$ &$\so(4)\oplus \C E\oplus \mathfrak{gl}(m+1)$
\\
\hline
&&&&\\[-3mm]

$\mathfrak{psl}(4|4)$&
\begin{tikzpicture}
\node[root]   (1)       [label=${{}^1}$]              {};
\node[xroot] (2) [right=of 1] {}[label=${{}^1}$] edge [-] (1);
\node[]   (3) [right=of 2] {$\;\cdots\,$} edge [-] (2);
\node[xroot]   (4) [right=of 3] {}[label=${{}^1}$] edge [-] (3);
\node[root]   (5) [right=of 4] {}[label=${{}^1}$] edge [-] (4);
\end{tikzpicture}&
$\C^ 4$ & $\mathbb{S}^+\otimes\C^{4}\oplus\mathbb{S}^-\otimes(\C^{4})^{*}$ &$\so(4)\oplus \C E\oplus \mathfrak{sl}(4)$
\\
\hline
&&&&\\[-3mm]
$\mathfrak{osp}(1|4)$&
\begin{tikzpicture}
\node[root]   (3) [label=${{}^2}$]      {} ;
\node[broot]   (4) [right=of 3] {} [label=${{}^2}$]     edge [-] (3);
\end{tikzpicture}&
$\C^3$ & $\mathbb{S}$ & $\so(3)\oplus\C E$
\\
\hline
&&&&\\[-3mm]
$\begin{gathered}\mathfrak{osp}(2m+1|4)\\ m\geq1\end{gathered}$&
\begin{tikzpicture}
\node[root]   (3) [label=${{}^2}$] {};
\node[xroot]   (4) [right=of 3] {} [label=${{}^2}$]  edge [-] (3);
\node[]   (5) [right=of 4] {$\;\cdots\,$} edge [-] (4);
\node[root]   (7) [right=of 5] {} [label=${{}^2}$]  edge [rdoublearrow] (5);
\end{tikzpicture}&
$\C^3$ & $\mathbb{S}\otimes\C^{2m+1}$ & $\so(3)\oplus\C E\oplus\so(2m+1)$
\\
\hline 
&&&&\\[-3mm]
$\mathfrak{osp}(2|4)$&
\begin{tikzpicture}
\node[root]   (3) [label=${{}^2}$] {} ;
\node[xroot]   (4) [above right=of 3] [label=right:${{}^1}$] {} edge [-] (3);
\node[xroot]   (5) [below right=of 3] [label=right:${{}^1}$] {} edge [-] (3) edge [doublenoarrow] (4);
\end{tikzpicture}&
$\C^3$ & $\mathbb{S}\otimes\C^2$ & $\so(3)\oplus\C E \oplus\so(2)$
\\
\hline 
&&&&\\[-3mm]
$\begin{gathered}\mathfrak{osp}(2m|4)\\ m\geq2\end{gathered}$&
\begin{tikzpicture}
\node[root]   (3)[label=${{}^2}$] {} ;
\node[xroot]   (4) [right=of 3][label=${{}^2}$] {} edge [-] (3);
\node[]   (6) [right=of 4] {$\;\cdots\quad$} edge [-] (4);
\node[root]   (7) [above right=of 6] [label=${{}^1}$]{} edge [-] (6);
\node[root]   (8) [below right=of 6] [label=${{}^1}$]{} edge [-] (6);
\end{tikzpicture}&
$\C^3$ & $\mathbb{S}\otimes\C^{2m}$ & $\so(3)\oplus\C E \oplus\so(2m)$
\\
\hline
&&&&\\[-3mm]
$\begin{gathered}\mathfrak{osp}(8|2n)\\ n\geq1\end{gathered}$&
\begin{tikzpicture}
\node[root]   (1)    [label=${{}^1}$]                 {};
\node[root] (2) [right=of 1] [label=${{}^2}$]{} edge [-] (1);
\node[root] (3) [right=of 2] [label=${{}^2}$]{} edge [-] (2);
\node[xroot] (4) [right=of 3] [label=${{}^2}$]{} edge [-] (3);
\node[]   (5) [right=of 4] {$\;\cdots\,$} edge [-] (4);
\node[root]   (6) [right=of 5] [label=${{}^1}$]{} edge [doublearrow] (5);
\end{tikzpicture}&
$\C^6$ & $\mathbb{S}^+\otimes \C^{2n}$ & $\so(6)\oplus\C E\oplus \mathfrak{sp}(2n)$
\\
\hline
&&&&\\[-3mm]
$\mathfrak{ab}(3)$&
\begin{tikzpicture}
\node[root]   (1)        [label=${{}^1}$]             {};
\node[xroot] (2) [right=of 1] {} [label=${{}^2}$]  edge [rtriplearrow] (1) edge [-] (1);
\node[root]   (3) [right=of 2] {} [label=${{}^3}$] edge [-] (2);
\node[root]   (4) [right=of 3] {} [label=${{}^2}$] edge [doublearrow] (3);
\end{tikzpicture}&
$\C^5$ & $\mathbb{S}\otimes\C^2$ & $\so(5)\oplus \C E\oplus\mathfrak{sl}(2)$\\
\hline 
\end{tabular}
\vskip0.7cm
\caption[]{\label{primatab}}
\end{sidewaystable}

\begin{proof}
First one looks at the even part $\s_{\bar0}$ of the $\Z$-graded Lie superalgebra $\s$. It is a reductive Lie algebra $\s_{\bar0}=[\s_{\bar0},\s_{\bar0}]\oplus\mathfrak{z}$ where $[\s_{\bar0},\s_{\bar0}]=\bigoplus_i\s^i$ is a direct sum of simple ideals and $\mathfrak z$ is the center of $\s_{\bar0}$. 

The $\Z$-grading of $\s$ induces $\Z$-gradings $\s^i=\s^i_{-2}\oplus\s^i_{0}\oplus\s^i_{2}$ on each simple factor and on the center $\mathfrak z=\bigoplus_{p\geq -1}\mathfrak z_{2p}$. 

By hypothesis (3), $\mathfrak z_{-2}=0$. 
The invariant bilinear form $B$ of $\s$ is non-degenerate on $\s_{\bar0}$ and hence on each $\s^i$ and on $\mathfrak z$. In all cases where $\mathfrak z$ is not zero, $B$ coincides with the Killing form of $\s$ (this follows from the discussion at the beginning of \S\ref{ebbeneebbene}, together with a direct inspection of Table \ref{Kactab}). It follows that the $\mathfrak z_p$ and $\mathfrak z_{-p}$ are dual to each other, and then $\mathfrak z=\mathfrak z_0$.

Let $D=\dim V$. By hypothesis (1) and (3), one can assume without loss of generality that $\s_{-2}\subset\s^1$ and $\s^i\subset\s_0$ for all $i\geq2$. In particular, the ideal $\so(V)\simeq \so(D)$ of $\s_0$ is contained in $\s^1$. 
The ideal $\s^1$ of $\s_{\bar0}$ has then a $\Z$-grading 
$$\s^1=\s^1_{-2}\oplus\s^1_0\oplus\s^1_2\,,$$
where
$$\s^1_{-2}=V\,\,,\,\,\s^1_0=(\so(V)\niplus\mathfrak e)\,\,,\,\,\s^1_{2}= V^*$$ 
and $\mathfrak e$ is an ideal of $\s^1_0$. It follows that $\s^1_0\simeq\mathfrak{co}(V)$ and $\s^1\simeq\so(D+2)$.

From the description of the even part of the basic Lie superalgebras given in Table \ref{Kactab}, $\s$ must be one of the following Lie superalgebras: $\mathfrak{osp}(2m+1|4)$, $\mathfrak{osp}(2|4)$ or $\mathfrak{osp}(2m|4)$ (for $m\geq 2$) for $D=3$; $\mathfrak{sl}(m+1|4)$ (for $m\neq 3$) or $\mathfrak{psl}(4|4)$ for $D=4$; $\mathfrak{ab}(3)$ for $D=5$; $\mathfrak{osp}(2m+1|2n)$ (for $m\geq 1$, $n\geq 1$) for $D=2m-1$; $\mathfrak{osp}(2m|2n)$ (for $m\geq 2$, $n\geq 1$) for $D=2m-2$. 

By hypothesis (4), the representation of $\so(V)$ on the odd part $\s_{\bar 1}$ contains at least one factor of half-spin type. This implies that also the representation of $\s^1\simeq\so(D+2)$ on $\s_{\bar 1}$
contains a factor of half-spin type (in the case $D=6$, since the semi-spinor and the tautological representations of $\so(8)$ are related by triality, this condition must hold true for \emph{some} identification $\s^1\simeq \so(8)$). 

By looking at Table \ref{Kactab}, one obtains exactly the simple Lie superalgebras $\s$ listed in the first column of Table \ref{primatab}.
\smallskip
\par
To conclude the proof, one first determines all $\mathbb{Z}$-gradings of depth $2$ of the $\s$ in the above list which satisfy hypotheses (1)--(3) and such that $$[\s_{0},\s_{0}]=\s^1_0\oplus\bigoplus_{i\geq 2}\s^i\ .$$
A case by case analysis of \cite[Tables 1-5]{Leites} reveals that the Dynkin diagrams satisfying the previous conditions are exactly those displayed in Table \ref{primatab}.

Finally, by using the explicit description in \cite{Kac77a} of the root systems associated to  the Dynkin diagrams of Table \ref{primatab}, it is a tedious but straightforward task to check that all the listed $\Z$-gradings also satisfy hypothesis (4). 
\end{proof}
We now explicitly describe the negatively graded parts $\s_{-2}\oplus\s_{-1}$ of the $\mathbb{Z}$-graded Lie superalgebras listed in Table \ref{primatab}.
In all cases except $\mathfrak{osp}(8|2n)$, the negatively graded part is isomorphic to a supertranslation algebra. 

For any Lie superalgebra $\s$ of Table \ref{primatab} (except $\mathfrak{osp}(8|2n)$), we now exhibit a non-degenerate bilinear form $\mcB$ on $W=\mathbb S\oplus\cdots\oplus\mathbb S$ satisfying
\begin{enumerate}
\item[(B1)]
$\mcB$ is $\so(V)$-invariant,
\item[(B2)]
$\mcB$ is symmetric ($\epsilon=1$) or anti-symmetric ($\epsilon=-1$), 
\item[(B3)]
for all $v\in V$ and $s,t\in W$, $\mcB(v\cdot s,t)=\epsilon\mcB(s,v\cdot t)$
\end{enumerate}
and such that the corresponding supertranslation algebra $\m=V\oplus W$ with the bracket \eqref{equazione_bracket} is identifiable with the negatively graded part $\s_{-2}\oplus\s_{-1}$ of $\s$.
\vskip0.3cm\par
\begin{example}$\mathfrak{sl}(m+1|4)$ (for $m\neq 3$) and $\mathfrak{psl}(4|4)$.\vskip0.3cm\par
\label{special}
As an $(\so(4)\oplus\mathfrak{sl}(m+1))$-module, $\s_{-1}$ is isomorphic to 
$$(\mathbb{S^+}\otimes \mathbb{C}^{m+1})\oplus(\mathbb{S^-}\otimes (\mathbb{C}^{m+1})^{*})$$
and the bracket $S^2(\s_{-1})\rightarrow V$ is induced by an $(\so(4)\oplus\mathfrak{sl}(m+1))$-equivariant map
$$
(\mathbb{S^+}\otimes \mathbb{C}^{m+1}) \otimes (\mathbb{S^-}\otimes (\mathbb{C}^{m+1})^{*})\rightarrow V\ .
$$
Hence, there exists an $\so(4)$-equivariant map $\tilde\Gamma:\mathbb{S^+}\otimes \mathbb{S^-}\rightarrow V$ such that
\[
[s^+\otimes c,s^-\otimes c^*]=\tilde\Gamma(s^+,s^-)c^*(c)\,,
\]
for any $s^{\pm}\in\mathbb{S^\pm}$, $c\in\mathbb{C}^{m+1}$ and $c^*\in(\mathbb{C}^{m+1})^{*}$.
By the results of \cite{AC}, the map $\tilde\Gamma$ is uniquely determined by a non-degenerate bilinear form $b:\mathbb{S}\otimes\mathbb{S}\rightarrow\mathbb{C}$ with invariants $(\tau,\sigma)=(-,-)$ via the usual formula $(\tilde\Gamma(s^+,s^-),v)=b(v\cdot s^+,s^-)$.

Identifying $\mathbb{C}^{m+1}$ and $(\mathbb{C}^{m+1})^{*}$, using a non-degenerate symmetric bilinear form $\delta$ on $\mathbb{C}^{m+1}$,
we define a Clifford multiplication $\s_{-2}\otimes \s_{-1}\rightarrow \s_{-1}$ by
\begin{equation}
\label{CliffA}
v\cdot (s\otimes c)=(v\cdot s)\otimes c\ ,
\end{equation}
and a non-degenerate $\so(4)$-invariant bilinear form $\mcB:\s_{-1}\otimes\s_{-1}\rightarrow\mathbb{C}$ by
\[
\mcB(s\otimes c,t\otimes d)=b(s,t)\delta(c,d)\ .
\]
Direct computations show that the Clifford multiplication \eqref{CliffA} and $\mcB$ satisfy conditions (B1), (B2), (B3).
\end{example}
\vskip0.3cm\par

\begin{example}$\mathfrak{osp}(2m+1|4)$ (for $m\geq 0$) and $\mathfrak{osp}(2m|4)$ (for $m\geq 1$).\vskip0.3cm\par
\label{orto}
This is the orthosymplectic case $\s=\mathfrak{osp}(N|4)$ with $N\geq 1$. In this case, $$\s_{-1}\simeq\mathbb{S}\otimes \mathbb{C}^{N}$$ as an $(\so(3)\oplus\mathfrak{so}(N))$-module and
the bracket is given by an $(\so(3)\oplus\mathfrak{so}(N))$-equivariant map
$
S^2(\mathbb{S})\otimes S^2(\mathbb{C}^{N})\rightarrow V
$.

There exists an $\so(3)$-equivariant map $\tilde\Gamma:S^2(\mathbb{S})\rightarrow V$ and a non-degenerate symmetric $\mathfrak{so}(N)$-invariant bilinear form $\delta$ on $\mathbb{C}^{N}$, given by the matrix with anti-diagonal entries equal to $1$,
such that
\[
[s\otimes c,t\otimes d]=\tilde\Gamma(s,t)\delta(c,d)\ ,
\]
for any $s,t\in\mathbb{S}$ and $c,d\in\mathbb{C}^{N}$. The map $\tilde\Gamma$ is uniquely determined by a non-degenerate bilinear form $b:\mathbb{S}\otimes\mathbb{S}\rightarrow\mathbb{C}$ with invariants $(\tau,\sigma)=(-,-)$, via the usual formula $(\tilde\Gamma(s,t),v)=b(v\cdot s,t)$.

We again define a Clifford multiplication by \eqref{CliffA} together with a non-degenerate $(\so(3)\oplus\mathfrak{so}(N))$-invariant bilinear form $\mcB=b\otimes \delta:\s_{-1}\otimes\s_{-1}\rightarrow\mathbb{C}$ which satisfy conditions (B1), (B2), (B3).
\end{example}
\vskip0.3cm\par
\begin{example}$\mathfrak{ab}(3)$.\vskip0.3cm\par
\label{effe}
In this exceptional case, $$\s_{-1}\simeq\mathbb{S}\otimes\C^2$$ as an $(\so(5)\oplus\mathfrak{sl}(2))$-module and
the bracket is given by an $(\so(5)\oplus\mathfrak{sl}(2))$-equivariant map
$
\Lambda^2(\mathbb{S})\otimes\Lambda^2(\mathbb{C}^{2})\rightarrow V
$.

Hence, there exists an $\so(5)$-equivariant map $\tilde\Gamma:\Lambda^2(\mathbb{S})\rightarrow V$ and a non-degenerate $\mathfrak{sl}(2)$-invariant bilinear form $\omega$ on $\mathbb{C}^2$ such that
\[
[s\otimes c,t\otimes d]=\tilde\Gamma(s,t)\omega(c,d)\ ,
\]
for any $s,t\in\mathbb{S}$ and $c,d\in\mathbb{C}^2$. 
The map $\tilde\Gamma$ is uniquely determined by a non-degenerate bilinear form $b:\mathbb{S}\otimes\mathbb{S}\rightarrow\mathbb{C}$ with invariants $(\tau,\sigma)=(+,-)$, via the usual formula $(\tilde\Gamma(s,t),v)=b(v\cdot s,t)$. 

We again define a Clifford multiplication by \eqref{CliffA} together with a non-degenerate $(\so(5)\oplus\mathfrak{sl}(2))$-invariant bilinear form $\mcB=b\otimes \omega:\s_{-1}\otimes\s_{-1}\rightarrow\mathbb{C}$ which satisfy conditions (B1), (B2), (B3).
\end{example}
\vskip0.3cm\par
In the $\mathfrak{osp}(8|2n)$ case the negatively graded part of $\s$ is not a supertranslation algebra, but it nevertheless admits a similar description in terms of semi-spinor modules.
\vskip0.3cm\par
\begin{example}$\mathfrak{osp}(8|2n)$ (for $n\geq 1)$.\vskip0.3cm\par
In this case, $$\s_{-1}\simeq\mathbb{S}^+\otimes \mathbb{C}^{2n}$$ as an $(\so(6)\oplus\mathfrak{sp}(2n))$-module and
the bracket is given by an $(\so(6)\oplus\mathfrak{sp}(2n))$-equivariant map
$
\Lambda^2(\mathbb{S}^+)\otimes\Lambda^2(\mathbb{C}^{2n})\rightarrow V
$.

There exists an $\so(6)$-equivariant map $\tilde\Gamma:\Lambda^2(\mathbb{S}^+)\rightarrow V$ together with a non-degenerate $\mathfrak{sp}(2n)$-invariant bilinear form $\omega$ on $\mathbb{C}^{2n}$ such that
\[
[s^+\otimes c,t^+\otimes d]=\tilde\Gamma(s^+,t^+)\omega(c,d)\ ,
\]
for any $s^+,t^+\in\mathbb{S}^+$ and $c,d\in\mathbb{C}^{2n}$.
\label{notfit}
\end{example}
\subsection{The maximal transitive prolongation}\hfill
\vskip0.3cm\par
So far we proved (Theorems \ref{nonloso}, \ref{findim}, \ref{classimple}) that the maximal transitive prolongation $\g$ of a supertranslation algebra $\m=V\oplus W$ with $\dim V\geq 3$ either satisfies $\g_p=0$ for all $p\geq 1$ or
\begin{enumerate}
\item[$-$] $\g$ is a finite-dimensional semisimple Lie superalgebra,
\item[$-$] $\g$ has a unique minimal ideal $\s$ which is a simple prolongation of $\m$,
\item[$-$] $\s$ is one of the $\Z$-graded Lie superalgebras listed in Table \ref{primatab}. 
\end{enumerate}
In this section, for each choice of $\s$ in Table \ref{primatab} we determine the corresponding maximal prolongation $\g$. It turns out that $\s=\g$ except in the case where $\s=\mathfrak{psl}(4|4)$ and $\g=\mathrm{der}(\mathfrak{psl}(4|4))\simeq\mathfrak{pgl}(4|4)$.

\begin{theorem}\label{teorematotale}
Let $\m=V\oplus W$ be a supertranslation algebra satisfying $\dim V\geq 3$, and $\g$ the maximal transitive prolongation of $\m$. If $\g_1\neq0$, then $\g$ is one of the Lie superalgebras listed in Table \ref{tabmax}. Therein, the symbol "$\cdots$" appearing in a Dynkin diagram of a Lie superalgebra denotes a subdiagram corresponding to a Lie algebra $\mathfrak{sl}(\ell+1)$ of appropriate (possibly zero) rank $\ell$.
\end{theorem}
{\small
\begin{table}[H]
\begin{centering}
\begin{tabular}{|c|c|c|c|c|c|}
\hline
$\g$ & Dynkin diagram & $\dim V$ & $\dim W$ & $N$ & $\h_0$\\
\hline
\hline
&&&&&\\[-3mm]

$\begin{gathered}\mathfrak{sl}(m+1|4)\\m\neq3\end{gathered}$&
\begin{tikzpicture}
\node[root]   (1)       [label=${{}^1}$]              {};
\node[xroot] (2) [right=of 1] {}[label=${{}^1}$] edge [-] (1);
\node[]   (3) [right=of 2] {$\;\cdots\,$} edge [-] (2);
\node[xroot]   (4) [right=of 3] {}[label=${{}^1}$] edge [-] (3);
\node[root]   (5) [right=of 4] {}[label=${{}^1}$] edge [-] (4);
\end{tikzpicture}&
$4$ & $4N$ & $m+1$ & $\mathfrak{gl}(m+1)$
\\
\hline
&&&&&\\[-3mm]

$\mathfrak{pgl}(4|4)$&
\begin{tikzpicture}
\node[root]   (1)       [label=${{}^1}$]              {};
\node[xroot] (2) [right=of 1] {}[label=${{}^1}$] edge [-] (1);
\node[]   (3) [right=of 2] {$\;\cdots\,$} edge [-] (2);
\node[xroot]   (4) [right=of 3] {}[label=${{}^1}$] edge [-] (3);
\node[root]   (5) [right=of 4] {}[label=${{}^1}$] edge [-] (4);
\end{tikzpicture}&
$4$ & $4N$ & $4$ & $\mathfrak{gl}(4)$
\\
\hline
&&&&&\\[-3mm]

$\mathfrak{osp}(1|4)$&
\begin{tikzpicture}
\node[root]   (3) [label=${{}^2}$]      {} ;
\node[broot]   (4) [right=of 3] {} [label=${{}^2}$]     edge [-] (3);
\end{tikzpicture}&
$3$ & $2$ & $1$ & $0$
\\
\hline
&&&&&\\[-3mm]

$\begin{gathered}\mathfrak{osp}(2m+1|4)\\ m\geq1\end{gathered}$&
\begin{tikzpicture}
\node[root]   (3) [label=${{}^2}$] {};
\node[xroot]   (4) [right=of 3] {} [label=${{}^2}$]  edge [-] (3);
\node[]   (5) [right=of 4] {$\;\cdots\,$} edge [-] (4);
\node[root]   (7) [right=of 5] {} [label=${{}^2}$]  edge [rdoublearrow] (5);
\end{tikzpicture}&
$3$ & $2N$ & $2m+1$ & $\so(2m+1)$
\\
\hline 
&&&&&\\[-3mm]

$\mathfrak{osp}(2|4)$&
\begin{tikzpicture}
\node[root]   (3) [label=${{}^2}$] {} ;
\node[xroot]   (4) [above right=of 3] [label=right:${{}^1}$] {} edge [-] (3);
\node[xroot]   (5) [below right=of 3] [label=right:${{}^1}$] {} edge [-] (3) edge [doublenoarrow] (4);
\end{tikzpicture}&
$3$ & $2N$ & $2$ & $\so(2)$
\\
\hline 
&&&&&\\[-3mm]

$\begin{gathered}\mathfrak{osp}(2m|4)\\ m\geq2\end{gathered}$&
\begin{tikzpicture}
\node[root]   (3)[label=${{}^2}$] {} ;
\node[xroot]   (4) [right=of 3][label=${{}^2}$] {} edge [-] (3);
\node[]   (6) [right=of 4] {$\;\cdots\quad$} edge [-] (4);
\node[root]   (7) [above right=of 6] [label=${{}^1}$]{} edge [-] (6);
\node[root]   (8) [below right=of 6] [label=${{}^1}$]{} edge [-] (6);
\end{tikzpicture}&
$3$ & $2N$ & ${2m}$ & $\so(2m)$
\\
\hline
&&&&&\\[-3mm]

$\mathfrak{ab}(3)$&
\begin{tikzpicture}
\node[root]   (1)        [label=${{}^1}$]             {};
\node[xroot] (2) [right=of 1] {} [label=${{}^2}$]  edge [rtriplearrow] (1) edge [-] (1);
\node[root]   (3) [right=of 2] {} [label=${{}^3}$] edge [-] (2);
\node[root]   (4) [right=of 3] {} [label=${{}^2}$] edge [doublearrow] (3);
\end{tikzpicture}&
$5$ & $4N$ & $2$ & $\mathfrak{sl}(2)$\\
\hline 
\end{tabular}
\end{centering}
\caption[]{\label{tabmax}}   
\vskip14pt\par
\end{table}}
\begin{proof}
By Theorem \ref{nonloso}, $\g$ is semisimple and contains a unique minimal ideal $\s$, which is a simple prolongation of $\m$ with $\so(V)\subset\s_0\subset\g_0$ and $\g_{\bar1}=\s_{\bar1}$. By Theorem \ref{findim}, $\g$ and $\s$ are finite-dimensional. It follows that $\s$ is one of the Lie superalgebras in Table \ref{primatab}, different from $\mathfrak{osp}(8|2n)$.

Recall that the {\it socle} of $\g$ is the sum of all non-zero minimal ideals of $\g$, and it is proved in \cite{Kac77a, Cheng1995} that it is of the form $\bigoplus_{i=1}^s (\s^i\otimes\Lambda(\mathbb C^{m_i}))$, where $\s^i$ is a simple Lie superalgebra and $m_i$ a non-negative integer, for every $i=1,\ldots,s$. 

The socle of the maximal transitive prolongation $\g$ of a supertranslation algebra $\m$ which satisfies $\dim V\geq 3$ and $\g_1\neq 0$ equals $\s$, and then $s=1$ and $m_1=0$ in this case.

Moreover, from the characterization of semisimple Lie superalgebras in \cite{Kac77a, Cheng1995} it follows that
$\s\subset\g\subset\mathrm{der}\,\s$.

To conclude, observe that $\s=\mathrm{der}(\s)$ for all Lie superalgebras in Table \ref{primatab}, with the exception of $\mathrm{der}(\mathfrak{psl}(4|4))\simeq \mathfrak{pgl}(4|4)$ \cite{Kac77a}. It is straightforward to check that the negatively graded part of $\mathrm{der}(\mathfrak{psl}(4|4))$ coincides with $\m$ and that $\mathrm{der}(\mathfrak{psl}(4|4))$  is a transitive prolongation of $\m$.
\end{proof}

\section{The classification}
\label{sec3}
\label{sec4}
\label{subsub1}
In this section we explicitly describe all the maximal transitive prolongations of supertranslation algebras, for all possible dimensions of $V$ and all $N\geq 1$. 
We include the cases $\dim V=1$ or $2$: Theorem \ref{findim} does not apply, and the maximal transitive prolongation turns out to be infinite-dimensional.

In the next Theorem, we denote by $K(m|n)$ the infinite-dimensional contact superalgebra in dimension $(m|n)$, with $m=2k+1$  (see, e.g., \cite{Kac98} for more details). By \cite[Prop. 3.1.3]{CK99}, $K(m|n)$ with its principal $\mathbb{Z}$-grading is the maximal transitive prolongation of its negatively graded part $$K(m|n)_{-2}\oplus K(m|n)_{-1}\,,$$
where
$$
K(m|n)_{-2}\simeq \C\,\,\,\,\text{and}\,\,\,\,K(m|n)_{-1}\simeq\C^{2k}\oplus\Pi\C^{n}\ .
$$
It is a simple Lie superalgebra.
\begin{theorem}
\label{N1c}
Let $V$ be a complex vector space with a non-degenerate symmetric bilinear form, $\mathfrak{m}=V\oplus W$ a supertranslation algebra and $\g=\bigoplus_{p\in\Z} \g_p$ the maximal transitive prolongation of $\m$.

If $\dim V=1$ or $2$, then $\g$ is infinite-dimensional and  
\begin{enumerate}
\item[$-$] $\mathfrak{g}=K(1|N)$ if $\dim V=1$,
\item[$-$] $\mathfrak{g}=K(1|N)\oplus K(1|N)$ if $\dim V=2$. 
\end{enumerate}
If $\dim V\geq 3$, then $\g$ is finite-dimensional, $\mathfrak{g}_0$ is as in Theorem \ref{lo0}, and either $\g_p=0$ for all $p \geq 1$ or
\begin{itemize}
\item[$-$] $\g=\mathfrak{osp}(N|4)$, $\dim V=3$ with  $\m$ as in  Example \ref{orto},
\item[$-$] $\mathfrak{g}=\mathfrak{pgl}(N|4)$, $\dim V=4$ with $\m$ as in  Example \ref{special},
\item[$-$] $\mathfrak{g}= \mathfrak{ab}(3)$, $\dim V=5$ with $\m$ as in  Example \ref{effe},
\end{itemize}
and with $\Z$-gradings as described in Table \ref{tabmax}. \\
\end{theorem}
\begin{proof}
The cases $\dim V\geq 3$ are a direct consequence of Theorem \ref{lo0} and Theorem \ref{teorematotale}. One needs to prove the statement for $\dim V=1, 2$.

If $\dim V=1$, one has 
$$\mathfrak{m}_{-1}=\Pi\mathbb{C}^{N}\,\,\,\,\text{and}\,\,\,\,\mathfrak{m}_{-2}=\mathbb{C}\ .$$ 
Clearly $\m\simeq K(1|N)_{-2}\oplus K(1|N)_{-1}$, thus $\mathfrak{g}\simeq K(1|N)$.

If $\dim V=2$, one can identify $\so(2)=\C$ and there exist decompositions
$$\m_{-1}=(\Pi\C^{N})_+\oplus(\Pi\C^{N})_-\,\,\,\,\text{and}\,\,\,\,\m_{-2}=\C_{+}\oplus\C_{-}\,\,,$$ with action of $\so(2)$ given by:
\begin{align*}
\lambda \cdot s_\pm=\pm\lambda s_\pm,\quad\lambda\cdot v_\pm=\pm2\lambda v_\pm\quad (\lambda\in\mathfrak{so}(2),\ s_\pm\in (\Pi\C^{N})_\pm,\ v_\pm\in \C_\pm).
\end{align*}
Then $\m=\m_+\oplus\m_-\simeq K(1|N)_{<0}\oplus K(1|N)_{<0}$ as a direct sum of ideals. By \cite[Prop. 3.3]{MN}, whose proof remains unchanged in the Lie superalgebra case, the statement follows also for $\dim V=2$.
\end{proof}

\section{Comparison with the Lie algebra case}
\label{concl}
We classified the maximal transitive prolongations $\g$ of supertranslation algebras $\m$ in Theorem \ref{N1c}. If $\dim V\geq 3$, the Lie superalgebra $\g$ is finite-dimensional and either $\g_p=0$ for all $p\geq 1$ or $\g$ is isomorphic to $\mathfrak{osp}(N|4)$, $\mathfrak{pgl}(N|4)$, $\mathfrak{ab}(3)$.

The analogous problem in the Lie algebra setting was solved in \cite{AlSa}. In that case, the dimension of the pseudogroup of automorphisms of a manifold $M$ endowed with an extended Poincar\'e structure $\mathcal{D}$ is bounded by $\dim \g$ and equality is obtained precisely when $M$ is locally isomorphic to the maximally homogeneous model $\overline{\mathbb M}$ \cite{N2, AS}. 

Tanaka's results have never been proved in the superalgebra setting, although it is plausible that an appropriate version should hold true for supermanifolds.
One of the problems is the lack of an estabilished notion of ``super-pseudogroup of automorphisms of a supermanifold''.
Hence, the geometric implications of finite-dimensionality of $\g$ must be considered rigorously true only at an infinitesimal level.
\medskip

Table \ref{tabellina} contains the list of maximal prolongations $\g$ of extended translation algebras with $\dim V\geq 3$ and $\g_{1}\neq 0$ \cite[Theorem 3.1]{AlSa}.
Comparison with Table \ref{tabmax} reveals very clear analogies with the Lie superalgebra case for $\dim V=3,4$. The analogy extends to the Lie algebra $\mathrm{F}_4$ and the Lie superalgebra $\mathfrak{ab}(3)$, however in this case the dimension of $V$ differs.

{\small
\begin{table}[H]
\begin{centering}
\begin{tabular}{|c|c|c|c|c|c|}
\hline
$\g$ & Dynkin diagram & $\dim V$ & $\dim W$ & $N$ & $\h_0$\\
\hline
\hline
&&&&&\\[-3mm]

$\begin{gathered}\mathfrak{sl}(m+1)\\ m\geq4\end{gathered}$&\begin{tikzpicture}
\node[root]   (1)                     {};
\node[crossroot] (2) [right=of 1] {} edge [-] (1);
\node[]   (3) [right=of 2] {$\;\cdots\,$} edge [-] (2);
\node[crossroot]   (4) [right=of 3] {} edge [-] (3);
\node[root]   (5) [right=of 4] {} edge [-] (4);
\end{tikzpicture}&
$4$ & $4N$ & $m-3$ &$\mathfrak{gl}(m-3)$
\\
\hline
&&&&&\\[-3mm]

$\begin{gathered}\mathfrak{sp}(2m)\\ m\geq3\end{gathered}$&
\begin{tikzpicture}
\node[root]   (1)                     {};
\node[crossroot] (2) [right=of 1] {} edge [-] (1);
\node[]   (3) [right=of 2] {$\;\cdots\,$} edge [-] (2);
\node[root]   (5) [right=of 3] {} edge [doublearrow] (3);
\end{tikzpicture}&
$3$ & $2N$ & $2(m-2)$ & $\mathfrak{sp}(2(m-2))$
\\
\hline
&&&&&\\[-3mm]

$\mathrm F_4$&
\begin{tikzpicture}
\node[root]   (1)                     {};
\node[root] (2) [right=of 1] {} edge [-] (1);
\node[root]   (3) [right=of 2] {} edge [rdoublearrow] (2);
\node[crossroot]   (4) [right=of 3] {} edge [-] (3);
\end{tikzpicture}&
$7$ & $8$ & $1$ & $0$
\\
\hline 
&&&&&\\[-3mm]

$\mathrm E_6$&
\begin{tikzpicture}
\node[crossroot]   (1)                     {};
\node[root] (2) [right=of 1] {} edge [-] (1);
\node[root]   (3) [right=of 2] {} edge [-] (2);
\node[root]   (4) [right=of 3] {} edge [-] (3);
\node[crossroot]   (5) [right=of 4] {} edge [-] (4);
\node[root]   (6) [below=of 3] {} edge [-] (3);
\end{tikzpicture}&
$8$ & $16$ & $1$ & $\mathfrak{so}(2)$
\\
\hline
\end{tabular}
\end{centering}
\label{tabellina}
\vskip14pt\par
\end{table}}
It would be interesting to find a $\Z$-graded contragredient Lie superalgebra in correspondence with the Lie algebra $\mathrm{E_6}$ of Table \ref{tabellina} and to look for applications to supergravity theories.
We remark that it cannot be a finite-dimensional or (twisted) affine Lie superalgebra, as it follows from the classification of contragredient Lie superalgebras
of finite Gelfand-Kirillov growth, see \cite{vdl, HS}. Dynkin diagrams with shape $\mathrm{E}_6$ appeared in the context of almost affine Kac-Moody Lie superalgebras in \cite{Leites}
but no extensive review of their properties is known to us.

\end{document}